\documentclass[reqno,11pt]{amsart}
\usepackage{geometry}
\usepackage{mathtools,amssymb,amsthm,mathrsfs,color,lineno,paralist,graphicx,float}
\usepackage[colorlinks,
linkcolor=blue,
anchorcolor=green,
citecolor=blue, 
]{hyperref}
\usepackage{tikz}
\usepackage{graphicx}

\usepackage{enumitem}
\usepackage[T1]{fontenc}
\usepackage[utf8]{inputenc}
\usepackage{subfig} 
\newsubfloat{figure}
\usepackage[justification = centering, labelsep =period]{caption} 

\usepackage{calc}
\definecolor{bleu1}{RGB}{0,57,128}
\def\bleu1{\color{bleu1}}

\usepackage{etoolbox}
\patchcmd{\section}{\normalfont}{\normalfont \bleu1}{}{}
\patchcmd{\subsection}{\normalfont}{\normalfont \bleu1}{}{}
\patchcmd{\subsubsection}{\normalfont}{\normalfont \bleu1}{}{}

\newtheorem{proposition}{Proposition}[section]
\newtheorem{theorem}{Theorem}[section]
\newtheorem{definition}{Definition}[section] 
\newtheorem{lemma}{Lemma}[section]
\newtheorem{remark}{Remark}[section]
\newtheorem{corollary}{Corollary}[section]

\newtheorem{example}{Example}

\newcommand{\Z}{{\mathbb Z}}
\newcommand{\C}{{\mathbb C}}
\newcommand{\R}{{\mathbb R}}

\newcommand{\T}{{\mathbb T}}

\usepackage{tikz,tikz-3dplot}
\tdplotsetmaincoords{80}{45}
\tdplotsetrotatedcoords{-90}{180}{-90}
\tikzset{surface/.style={draw=blue!70!black, fill=blue!40!white, fill opacity=.6}}

\usepackage{pgfplots}
\usepgfplotslibrary{polar}
	\usepgfplotslibrary{polar}
\pgfplotsset{compat=1.17}
\makeatletter
\tikzset{reuse path/.code={\pgfsyssoftpath@setcurrentpath{#1}}}
\makeatother

\begin{document}

	\title[]{Non-Self-Adjoint Quasi-periodic Operators with complex spectrum}

\author{Zhenfu Wang}
\address{
	Chern Institute of Mathematics and LPMC, Nankai University, Tianjin 300071, China
}

\email{zhenfuwang@mail.nankai.edu.cn}

\author {Jiangong You}
\address{
Chern Institute of Mathematics and LPMC, Nankai University, Tianjin 300071, China} \email{jyou@nankai.edu.cn}

\author{Qi Zhou}
\address{
Chern Institute of Mathematics and LPMC, Nankai University, Tianjin 300071, China
}

\email{qizhou@nankai.edu.cn}

	\begin{abstract}
	 We  give a precise and complete description on the spectrum for   a class of non-self-adjoint quasi-periodic operators  acting on $\ell^2(\mathbb{Z}^d)$ which contains the Sarnak's model as a special case.
	As a consequence, one can see various interesting spectral phenomena 	including  $\mathscr{P}\mathscr{T}$ symmetric breaking, the non-simply-connected two-dimensional spectrum in this class of operators.  Particularly, we provide new examples of non-self-adjoint operator in $\ell^{2}(\Z)$ whose spectra (actually a two-dimensional subset of $\mathbb{C}$) can not be approximated by the spectra of its finite-interval truncations.
		
			\end{abstract}
	\maketitle

	\section{Introduction}
	Due to its significance in open quantum systems, non-self-adjoint quasi-periodic operators (non-Hermitian quasicrystals  in physics literature)  have received more attention recently in physics. It has been  demonstrated that certain phenomena, such as topological phase transitions \cite{LZC,Longhi}, the coexistence of complex and real spectra \cite{Liu}, and skin effects \cite{skin,LZC}, are only observed in the non-self-adjoint  operators.  Most of these phenomena have not been rigorously justified at the mathematical level
	 because many of the methods used in self-adjoint operators do not work, especially due to a lack of the spectral theorem in the non-self-adjoint setting.
	
	The study of  non-self-adjoint quasi-periodic operators has also been shown to be important in mathematics. For instance, 
Avila \cite{Av0} established the  ground-breaking  global theory of  one-frequency quasi-periodic Schr\"odinger operators. To
 establish  its quantitative version\cite{GJYZ},  the core is to develop the  Thouless-type  formula for the non-self-adjoint operator:
	\begin{equation*}
		(H\psi)({n})=\psi({n+1})+\psi({n-1})+v(\theta+n\alpha+ {i}\epsilon )\psi({n}),
	\end{equation*}
and its dual operator.    The spectrum of non-self-adjoint  Schr\"odinger  operators also has a  deep connection with problems of the elliptic operators, such as ground states, steady states and averaging theory \cite{Kozlov1984,LWZ} and  non-Hermitian quasicrystals \cite{WWYZ}.

	The earliest study of non-self-adjoint quasi-periodic operators can be traced back to Sarnak \cite{Sarnak}, who studied the example
	\begin{equation}\label{operator}
		H_{\alpha}(\lambda,\omega)=\Delta+\lambda e^{2 \pi i(\omega+\langle n, \alpha\rangle)} \delta_{nn'}, \quad n \in \mathbb{Z}^d, 
	\end{equation}
	where $\Delta$ denotes the usual Laplacian on the $\mathbb{Z}^d$ lattice, $\lambda\in\mathbb{C}$ is the coupling constant, $\omega \in \mathbb{T} =\mathbb{R}/\mathbb{Z}$ is the phase, and the frequency $\alpha\in \mathbb{T}^{d}$ is  rationally independent, i.e.,  $(1,\alpha)$ is independent  over $\mathbb{Q}$.
	
	The spectrum of the operator defined by (\ref{operator}) has already been completely understood: 	
	\begin{theorem}\cite{Boca,B,Sarnak,xue}\label{app1}
		Assume that $\alpha \in \mathbb{T}^d$ is rationally independent, $\omega\in\mathbb{T}$ and $\lambda \in \mathbb{C}$. Let $$G_d(z)=\int_{\mathbb{T}^d} \log |z-\sum_{j=1}^{d}\cos2\pi\theta_j| \mathrm{d} \theta.$$
		Then the spectrum of (\ref{operator}) is
		$$
		\sigma(H_{\alpha}(\lambda,\omega))=\{z \in \mathbb{C}\big|G_d(z)=\log|\lambda|\}\cup\{z \in [-2d,2d]\big|G_d(z)>\log|\lambda|\} 
		$$
		for all $  \omega\in\mathbb{T}$.
	\end{theorem}
	\begin{remark}
		Sarnak originally proved Theorem \ref{app1} for the case where $|\lambda|\neq e^{G_d(0)}$ and  $\alpha$ is Diophantine \cite{Sarnak}. Boca later generalized the result for all $\lambda\in\mathbb{C}$ and rationally independent frequency \cite{Boca}. In the case that $d=1$, Borisov-Fedotov used a renormalization technique to describe the geometry of the spectrum  \cite{B}, while Wang-You-Zhou provided a proof using Avila’s global theory \cite{xue}. 
	\end{remark}
	\subsection{Main results}

To see richer spectrum phenomena,  we consider the  long-range operators 
	\begin{equation}\label{zd}
	(H^V_{\alpha}(\lambda,\omega) u)( n)=\sum_{k\in\mathbb{Z}^d}v_ku(n-k)+\lambda e^{2\pi i(\omega+\langle n,\alpha\rangle)} u(n) , \quad n \in \mathbb{Z}^d,
\end{equation}
which include  Sarnak's model (\ref{operator}) as a special case.
  We mention that the operator  \eqref{zd} (in fact its dual model \eqref{dual}) has captured the interests of physicists \cite{Longhi2,Longhi1} as  a special model for  non-symmetric hopping, which is an important area of research \cite{HN,HN2,HN3,LZC}.

Let  $V(\theta)=\sum_{k\in\mathbb{Z}^d}v_ke^{2\pi i\langle k,\theta\rangle}\in C(\mathbb{T}^d,\mathbb{C})$
and $G_d(z)= \int_{\mathbb{T}^d} \log |z-V(\theta)| \mathrm{d} \theta$. We will assume that
		$\log|z-V(\theta)|\in L^1(\mathbb{T}^d)$ for all $z\in\mathbb{C}$.
This assumption  is not  restricted, in fact, it is satisfied when  $V(\cdot)$ is  an analytic function, a Morse function, or a smooth function satisfying the transversality condition  \cite{klein, Klein}:
$\text{For any } \theta\in\mathbb{T}^d, \text{ there is } m\in\mathbb{N}^d, |m|\neq0  \text{ such that }\partial^mV(\theta)\neq0$.

For each $\lambda\in \mathbb{C}\backslash\{0\}$,	 we denote 
	$$
	P_\lambda=\{z \in \mathbb{C}: G_d(z)=\log |\lambda|\} , \qquad  C_\lambda=\{z \in \mathbb{C}: G_d(z)> \log |\lambda|\} \cap R(V),
	$$
	where $R(V)$ is the range of $V(\cdot)$.    Let $S_\lambda=P_\lambda\cup C_\lambda $  and  $S_0=R(V)$, which are uniquely determined by   $V(\cdot)$. Now we state our main theorem:
	
		\begin{theorem}	\label{thm}
		Assume that $\alpha \in \mathbb{T}^d$ is rationally independent, and $\lambda \in \mathbb{C}$, $V(\cdot)\in C(\mathbb{T}^d,\mathbb{C})$ is H\"{o}lder continuous, and $\log|z-V(\cdot)|\in L^1(\mathbb{T}^d)$ for all $z\in\mathbb{C}$. Then the spectrum of 	$H^V_{\alpha}(\lambda,\omega)$ defined in (\ref{zd})
	is
		$$
		\sigma(H^V_{\alpha}(\lambda,\omega))=S_{\lambda}\quad\text{for all }                \omega\in\mathbb{T}.
		$$
	\end{theorem}

	\subsubsection{Two-dimensional spectrum}
	
		It is interesting to explore  various spectral phenomena that can not be observed in the self-adjoint case. Theorem \ref{thm} offers  criteria for the existence of a two-dimensional spectrum. 
		
		When $d\ge 2$ and $V(\cdot)$ is complex, typically $R(V)$ is a two-dimensional subset of $\mathbb{C}$. In this case, it is not difficult to get a two-dimensional spectrum from Theorem \ref{thm}.
\begin{corollary}\label{mfre}
	Let  $V(\theta)=V(\theta_1,\cdots,\theta_d)\in C^\omega(\mathbb{T}^d,\mathbb{C})$. 
	Suppose there exists $\tilde{\theta}\in\mathbb{T}^d$, such that 
	\begin{equation}\label{non-dege}rank(\frac{\partial(\Re V)}{\partial \theta_i}(\tilde{\theta}),\frac{\partial(\Im V)}{\partial \theta_i}(\tilde{\theta}))_{1\leq i\leq d}=2.\end{equation}
	Then if $|\lambda|<e^{\min_{z\in \C}G_d(z)}$, ${H}^V_{\alpha}(\lambda, \omega)$ exists two-dimensional spectrum.
\end{corollary}

\begin{remark}
	In case $d\ge 2$, the non-degeneracy condition \eqref{non-dege} is easily satisfied. We remark that this non-degeneracy condition is also necessary due to  Example \ref{two6} in Section \ref{exam}.	
\end{remark}

In case  $d=1$, the two-dimensional spectrum actually comes from $P_\lambda$. More precisely, we have the following criterion:	
	\begin{theorem}\label{two4}
		Let $V(\theta)=\sum\limits_{k=l}^{m} v_{k} {e}^{2\pi i k\theta}\in C^{\omega}(\mathbb{T},\mathbb{C})$  be non-constant with  
		\begin{equation}\label{con}\sum_{k=l,k\neq0}^{m-1}|v_k|<|v_m|. \end{equation} Then ${H}^V_{\alpha}(\lambda, \omega)$ has two-dimensional spectrum when $|\lambda|=|v_m|$.
	\end{theorem}

\begin{remark}
\begin{enumerate}
\item It is not easy to observe two-dimensional spectra in one-frequency quasi-periodic models.
 Moreover, this result   provides  new examples showing that the spectra of non-self-adjoint operators in $\ell^2(\mathbb{Z})$ may not be approximated by the spectra of its finite-truncation $H_{[0,N]}$, as the eigenvalues of $H_{[0,N]}$ always belong to $R(V)$.  This kind of result was only observed in the random Hatano-Nelson model \cite{Gold3} previously.
\item The assumption \eqref{con} is necessary due to Theorem \ref{app1} and  Corollary \ref{twod}, and Theorem \ref{two4}  is optimal  due to Example \ref{app2} in Section \ref{exam} if  $V(\theta)=e^{2\pi i\theta}$.
\item   
If  $V(\theta)=e^{2\pi i\theta}$, \cite{bell}  claimed
that  $\sigma({H}^V_{\alpha}(\lambda, \omega))$ is the unit disc  for $|\lambda|=1$, 
 which was proved in  \cite{Boca1,Fang,ZFS} using $C^*$ algebra technique.
  \end{enumerate}
\end{remark}

Further information concerning the location of the two-dimensional spectrum
 can also be provided. 
\begin{corollary}\label{shape}
	Let $V(\cdot)\in C(\mathbb{T},\mathbb{C})$ be H\"{o}lder continuous, $\log|z-V(\cdot)|\in L^1(\mathbb{T}^d)$ for all $z\in\mathbb{C}$.
	Suppose that  $\Sigma$ is the non-empty two-dimensional part of $\sigma({H}^V_{\alpha}(\lambda, \omega))$.  If $R(V)=\partial D$ for some bounded open set $D=\cup_{i=1}^{k}D_i$, where $D_i$ are all the connected components of $D$, then $\Sigma=\cup_{j=1}^{l}D_{j}$ for some $1\leq l\leq k.$ 
\end{corollary}
	
	\begin{remark}
	Example \ref{two5} in Section \ref{exam} actually shows that $\Sigma$ could not be all connected components of $D$.
	\end{remark}

	\subsubsection{$\mathscr{P}\mathscr{T}$ symmetric breaking} 
As shown in Sarnak's model given in \eqref{operator},  the spectrum is confined to an interval of $[-2d,2d]$ for $|\lambda|\leq e^{G_d(0)}$, and the complex spectrum appears when  $|\lambda|> e^{G_d(0)}$. We found that this phenomenon is typical if $v_k=\bar{v}_{-k}.$ 
	This spectral phenomenon is related to an important topic in physics \cite{bender1,bender}: $\mathscr{P}\mathscr{T}$ symmetric breaking. 
		Recall  that ${H}^V_{\alpha}(\lambda, \omega)$ is a $\mathscr{P}\mathscr{T}$ symmetric operator if and only if $\lambda\in\mathbb{R}$, $\omega=0$ and $v_k=\overline{v}_{-k}$ \cite{bender}. In the quasi-periodic setting  we say that the family ${H}^V_{\alpha}(\lambda, \omega)$ has $\mathscr{P}\mathscr{T}$ symmetry if ${H}^V_{\alpha}(|\lambda|, 0)$ is $\mathscr{P}\mathscr{T}$ symmetric, since the spectrum of ${H}^V_{\alpha}(\lambda, \omega)$ is independent of $\omega$ \cite{lim}.
		 $\mathscr{P}\mathscr{T}$ symmetric operators represent an important class of non-self-adjoint operators in quantum mechanics \cite{PT}. It is known that  $\mathscr{P}\mathscr{T}$ symmetric operators often have real spectra \cite{bender}. If a $\mathscr{P}\mathscr{T}$ symmetric operator yields complex spectrum, it is referred to as $\mathscr{P}\mathscr{T}$ symmetric breaking.  $\mathscr{P}\mathscr{T}$ symmetric breaking has been found in several quasi-periodic operators  \cite{ LW, LZC, Sarnak}. By Theorem \ref{thm}, one can see clearly that $\mathscr{P}\mathscr{T}$ symmetric breaking occurs in  the class of operators  (\ref{zd}).
	
	\begin{corollary}\label{pt}
		Let $V(\cdot)\in C^{\omega}(\mathbb{T}^d,\mathbb{R})$, the operator $H^V_{\alpha}(\lambda,\omega)$  is  $\mathscr{P}\mathscr{T}$ symmetric and display $\mathscr{P}\mathscr{T}$ symmetric breaking.  In particular, denote $$\lambda_0(V)=e^{\min_{z\in R(V)}G_d(z)}, \qquad \lambda_1(V)=e^{\max_{z\in R(V)}G_d(z)},$$  then we have
		\begin{enumerate}
			\item If $|\lambda|\leq {\lambda_0(V)}$, $	\sigma(H^V_{\alpha}(\lambda,\omega))=[\min\limits_{\theta\in\mathbb{T}^d}V(\theta),\max\limits_{\theta\in\mathbb{T}^d}V(\theta)]\subset\mathbb{R}$.
			\item If $\lambda_0(V)<|\lambda|\leq\lambda_1(V)$, $	\sigma(H^V_{\alpha}(\lambda,\omega)) $ has both complex spectrum and real spectrum. 
			\item If $|\lambda|>\lambda_1(V)$, $\sigma(H^V_{\alpha}(\lambda,\omega))$ is a real analytic curve of $\R^2$.
		\end{enumerate}
	\end{corollary}

\begin{remark}
 It is known that the   self-adjoint Schr\"odinger operator
	 \begin{equation*}
(H_{V} u)(n)= u(n+1)+u(n-1) + V (n) u(n)=z u(n),
\end{equation*}	
(i.e. $V(n)$ is real) is isospectral to $H_{0}$ if and only if \,$V\equiv 0$ \cite{Bo,KS}.
 However, this is not the case for  non-self-adjoint quasi-periodic operators. In the almost-periodic case, Wang-You-Zhou  \cite{xue} extended Sarnak's result, and showed that if  the potential $\lambda V(\cdot) $ satisfies some  cone structure, and the coupling constant $\lambda$ is small, then  its spectrum is $[-2,2]$. 
 	Corollary \ref{pt} further generalizes Sarnak's model to long-range operators and provides additional examples of interval spectrum. 
	\end{remark}
	
\subsubsection{Hatano-Nelson model}
As a special example, the Hatano-Nelson model merits separate consideration. 
This  is a single particle, non-Hermitian Hamiltonian defined by 
\begin{equation*}
	(H_{\omega} u)(n)= e^gu(n+1)+e^{-g}u(n-1) + V_{\omega} (n) u(n).
\end{equation*}
It was introduced to  study the
motion of magnetic flux lines in disordered type-II  superconductors \cite{HN,HN2,HN3}, and has been extensively studied in recent years due to its relevance in understanding the localization properties and  topological properties \cite{Gong,LW,LZC}. 

The random Hatano-Nelson model is defined as 
$\{V_\omega(n)\}_{n\in\mathbb{Z}}$
being i.i.d..
The spectrum of this model has been extensively studied. Goldsheid and his  collaborators  \cite{Gold1,Gold3,Gold2}  studied the asymptotic properties of the restriction of  $H_\omega$ to the interval $[1,N]$. Conversely, Davies \cite{Davies,Davies1} studied 
$H_\omega$
as an operator acting on $\ell^2(\Z)$ and concluded that 
if the support of  $V_\omega$ is $[- \lambda,\lambda]$, 
	 the spectrum of $H_\omega$  almost surely contains a two-dimensional subset.
More precisely,   for almost surely $\omega$ $$[-\lambda,\lambda]+\{e^{-g}e^{2\pi i\theta}+e^ge^{-2\pi i\theta}|\theta\in[0,1]\}\subset \sigma(H_\omega).$$  In particular, if $\lambda\geq e^g+e^{-g}$, $$\sigma(H_\omega)=[-\lambda,\lambda]+\{e^{-g}e^{2\pi i\theta}+e^ge^{-2\pi i\theta}|\theta\in[0,1]\}.$$ If $\lambda< e^g-e^{-g}$, $\sigma(H_\omega)$ has a hole $$\{z\big| |z|\leq e^g-e^{-g}-\lambda\}\not\subset\sigma(H_\omega).$$ However,  it is not clear whether the hole exists  when $e^g-e^{-g}\leq\lambda<e^g+e^{-g}$.

In contrast, in the quasi-periodic setting, we discover that the spectrum is typically one-dimensional except for one $\lambda$:
	\begin{corollary}\label{twod}
		Assume that $\alpha \in\mathbb{T}\backslash\mathbb{Q}$, $\omega\in\mathbb{T}$ and $g\in\mathbb{C}$ such that $\Re g\geq0$.	  Consider the operator
\begin{equation}\label{eg}
				(H_{\alpha}(\lambda,\omega) u)( n)=e^{-g}u(n-1)+e^{g}u(n+1)+\lambda e^{2\pi i(\omega+ n\alpha)} u(n) , \quad n \in \mathbb{Z}.
		\end{equation}
	Then
	$$\sigma({H}_{\alpha}(\lambda, \omega))=\begin{cases}
		\{e^{-g}e^{2\pi i\theta}+e^ge^{-2\pi i\theta}|\theta\in[0,1]\}, & |\lambda|<e^{\Re g},\\
		\{\eta e^{-g}e^{2\pi i\theta}+\frac{1}{\eta} e^ge^{-2\pi i\theta}|\eta\in[1,{e^{\Re g}}],\theta\in[0,1]\}, & |\lambda|=e^{\Re g},\\
		\{\frac{e^{\Re g}}{|\lambda|}e^{-g}e^{2\pi i\theta}+{e^{-\Re g}|\lambda|} e^ge^{-2\pi i\theta}|\theta\in[0,1]\},& |\lambda|>e^{\Re g},		
	\end{cases}
$$
	for all $ \omega\in\mathbb{T}$.
	
	\end{corollary}

  \begin{remark}
By a proper transformation,  the more general operators 
\begin{equation}\label{ab}(H_{\alpha}(\lambda,\omega) u)(n)=au(n-1)+bu(n+1)+\lambda e^{2\pi i(\omega+ n\alpha)} u(n)
\end{equation}
 where $a,b\in\mathbb{C}$ can be converted into operator (\ref{eg}). Furthermore,  the spectrum of  (\ref{ab}) is, in fact, a rotation and scaling of the spectrum of (\ref{eg}). In particular, if $g=\frac{\pi}{2}i$, the above model can be transformed to
 $$(H_{\alpha}(\lambda,\omega) u)(n)=-u(n-1)+u(n+1)+\lambda e^{2\pi i(\omega+ n\alpha)} u(n),$$ with  different sign in the hopping. This model has also been extensively studied \cite{Davies,Gong,Mar1,Mar}.  
  One can consult  \cite{WWYZ1}  for more facts about   general  quasi-periodic Hatano-Nelson models.
 
		\end{remark}

Hatano-Nelson also investigated high-dimensional model to examine  magnetic flux lines in superconductors   and observed a transition from localized states to delocalization  \cite{HN}. Higher dimensions offer a multifaceted environment that facilitates compelling interactions among qualitatively unique phenomena \cite{bha,Lee}.
The high-dimensional Hatano-Nelson model is particularly challenging, we take $\Z^2$ as an example, and consider 
$$(H_{\omega}u)(m,n) =  (H_0u)(m,n)+V_{\omega}(m,n) u(m,n)$$
where 
$$ (H_0u)(m,n)=e^{-g}u(m-1,n)+e^{g}u(m+1,n)+u(m,n-1)+u(m,n+1)$$ for some $g>0$. 
 In the random case, 
the spectrum of $H_\omega$ 
 	is almost surely two-dimensional \cite{Davies}. Specifically,  if the support of  $V_\omega$ is $[- \lambda,\lambda]$ 
 and $\lambda\geq e^{g}+e^{-g}-2,$ then $\sigma({H}_\omega)= E+[- \lambda,\lambda]$, where 
$$E=\{[-2,2]+e^{-g}e^{2 \pi i\theta}+e^{g}e^{-2\pi i\theta}|\theta\in[0,1]\}$$
is the spectrum of $H_0$. 

For high-dimensional quasi-periodic Hatano-Nelson model, we obtain the following:
		\begin{corollary}\label{twod1}
		Assume that  $\alpha=(\alpha_1,\alpha_2) \in\mathbb{T}^2$ is rationally independent, $\omega\in\mathbb{T}$ and $g>0$.	  Consider the operator  
		\begin{align*}\label{eg}
			(H_{\alpha}(\lambda,\omega) u)(m,n)= (H_0u)(m,n)+\lambda e^{2\pi i(\omega+ m\alpha_1+ n\alpha_2)} u(m,n) , \quad (m,n) \in \mathbb{Z}^2.
		\end{align*}
		Then
		$$\sigma({H}_{\alpha}(\lambda, \omega))=\begin{cases}
		E, & |\lambda|<e^{ g},\\
		\mathrm{Conv}(E), & |\lambda|=e^{ g},
		\end{cases}
		$$
		for all $ \omega\in\mathbb{T}$,  where $\mathrm{Conv}$ denotes the closed convex hull.
		Moreover, if $$\lambda>e^{\max_{|z|\leq 2(e^g+e^{-g}+2)}G_2(z)},$$  then $\sigma({H}_{\alpha}(\lambda, \omega))$ is a real analytic curve.
		
	\end{corollary}
	
		\begin{figure}[h!]
		\centering  	
		\subfloat[$|\lambda|<$$e^{ g}$]{
	\begin{tikzpicture}[scale=1.5]

		\draw[fill=gray!30,rounded corners=22pt]
		(-1,-0.5) rectangle ++(2,1);		
			\draw [fill=white](0, 0) circle (0.3);	
	 
	\end{tikzpicture}
		}
		\hspace{10pt}  
		\subfloat[$|\lambda|$=$e^{ g}$]{
			\begin{tikzpicture}[scale=1.5]
					\filldraw[fill=gray!30,rounded corners=22pt]
				(0,0) rectangle ++(2,1);
			\end{tikzpicture}
		}
	
		\caption{\quad Spectrum with a hole}
		\label{tu9}
	\end{figure}
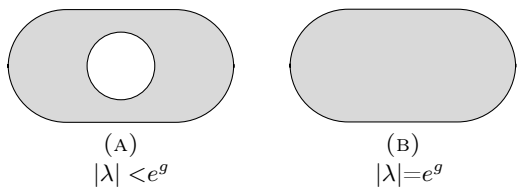

\begin{remark}
As depicted in Figure  \ref{tu9}, the spectrum of ${H}_{\alpha}(\lambda, \omega)$ has a hole  if $|\lambda|<e^g$. To the  best knowledge of the authors, this gives the first example of non-self-adjoint  quasi-periodic operator, which possesses non-simply-connected two-dimensional spectrum.   If $|\lambda|=e^g$, the hole   will be filled.  More surprisingly,   the spectrum becomes   an analytic curve for sufficiently large 
	$\lambda$. This is completely different from the behavior in the random case.\end{remark}

\subsection{Methods of the proof}
Let us explain the main ideas. The key observation is that the spectrum of $H^V_{\alpha}(\lambda,\omega)$ is same as its dual model  \cite{AA}:
\begin{equation}\label{dual}
	(\widehat{H}^V_{\alpha}(\lambda, \theta)u)(n)=\lambda u(n-1)+V(\theta+n\alpha)u(n), \quad n \in \mathbb{Z}.
\end{equation}
So the proof of Theorem \ref{thm}  is reduced to the following Theorem.
\begin{theorem}\label{first order}
 	Under the same assumptions on $V(\cdot)$ as in  Theorem \ref{thm}, the spectrum of \eqref{dual} is
$$
\sigma(\widehat{H}^V_{\alpha}(\lambda, \theta))=S_{\lambda}\quad\text{for all  }                \theta\in\mathbb{T}^d.
$$
\end{theorem}
\begin{remark}
If $\lambda=0$, it is easy to prove  that $\sigma(\widehat{H}^V_{\alpha}(\lambda, \theta))=S_{0}$, so we will always assume $\lambda\neq0.$ 
\end{remark}

Sarnak \cite{Sarnak} and Boca \cite{B} both used similar ideas to study (\ref{operator}) by exploring (\ref{dual}). 	As observed in \cite{B,Sarnak}, it is easy to establish 
$$
\sigma(\widehat{H}^V_{\alpha}(\lambda,\omega)) \subset S_{\lambda}=P_{\lambda}\cup C_{\lambda}.
$$
However, to prove the converse, particularly to prove  $C_{\lambda} \subset \sigma(\widehat{H}^V_{\alpha}(\lambda,\omega))$,  is  really challenging.  For this task, Sarnak \cite{Sarnak} used smooth approximation and relied on the Diophantine properties of frequency and distribution of zeros of analytic functions, while Boca's proof  \cite{B} utilizes $C^*$ algebra and relies on the profile of the spectrum. In this paper, we introduce a new method to prove $C_{\lambda} \subset 
\sigma(\widehat{H}^V_{\alpha}(\lambda,\omega)) $, which can be applied to more general operators.  More precisely,  we first establish the result  for Liouvillean frequency by modified Gordon's argument (Proposition \ref{gordon}), then we use the weak continuity of the spectrum (Proposition \ref{jian}) to show that the result holds for all rationally independent frequencies.  Notice that the classical Gordon's argument \cite{AYZ,Gordon}  is a main tool  for excluding the point spectra for self-adjoint  Schr\"odinger operators (see also Remark \ref{go-cla}). Here we give  a new application of it in non-self-adjoint operators. 
On the other hand, to obtain the weak continuity of the spectrum (Proposition \ref{jian}), one would follow the  classical  
argument that the spectrum is continuous in Hausdorff topology \cite{simon_vanmouche_avron,Ch} (in the self-adjoint case). However, one may meet the typical problem for the non-self-adjoint operator, that is 
\begin{equation}\label{inq}
\|(H-z)^{-1}\|\geq \frac{1}{dist(z,\sigma(H))}
\end{equation}
while in the self-adjoint case, we always have equality.  Previous work has developed the whole pesudo-spectrum theory \cite{Bott,Davies2,Tre1} to deal with this difficulty. However,  we will, avoiding \eqref{inq}, conclude the weak continuity argument by Weyl's criterion (for non-self-adjoint operators). 
 We remark that assuming only H\"{o}lder continuity for  potentials is an advantage of our argument.

	\subsection{Some explicit examples}\label{exam}
In this section, we provide some exactly solvable  models. We believe that these examples are of interest on their own,  but the initial purpose of finding them was to provide a natural completion for the study of \eqref{zd}.
Let us begin by considering the simplest case, where the Fourier coefficient of  $V(\cdot)$  only has one term:

	\begin{example}\label{app2}
		Assume that $\alpha \in\mathbb{T}\backslash\mathbb{Q}$ and $\omega\in\mathbb{T}$.	  Consider the operator
		\begin{equation*}
			(H_{\alpha}(\lambda,\omega) u)( n)=u(n-1)+\lambda e^{2\pi i(\omega+ n\alpha)} u(n) , \quad n \in \mathbb{Z}.
		\end{equation*}
		Then

		\begin{equation*}
			\sigma({H}^V_{\alpha}(\lambda, \omega))=
			\begin{cases}
				\mathbb{S}^1, & |\lambda|<1,\\
				\mathbb{D}, & |\lambda|=1,\\
				\lambda\mathbb{S}^1,& |\lambda|>1,		
			\end{cases}
		\end{equation*}
		for all $\omega\in\mathbb{T}.$

	\end{example}

		\begin{remark} 	 The spectrum profile in Example \ref{app2} is depicted in Figure \ref{tu1}.
We remark that the spectrum is not continuous in  the coupling constant $\lambda$, which reflects the spectrum of non-self-adjoint operator might be highly unstable.        	\end{remark}	

	\begin{figure}[h!]
		\centering  
		
		\subfloat[$|\lambda|$<1]{
			\begin{tikzpicture}[baseline,remember picture]
				\draw (0, 0) circle (1);		
			\end{tikzpicture}
		}
		\hspace{10pt}  	
		\subfloat[$|\lambda|$=1]{
			\begin{tikzpicture}[baseline,remember picture]
				\fill[gray!30] (0, 0) -- (1, 0) arc (0:180:1) -- (0, 0) -- (0, 0) arc (180:0:0);
				\fill[gray!30] (0, 0) -- (1, 0) arc (360:180:1) -- (0, 0) -- (0, 0) arc (180:360:0);	
				\draw (0, 0) circle (1);		
			\end{tikzpicture}
		}
		\subfloat[$|\lambda|$>1]{
			\begin{tikzpicture}[baseline,remember picture]
				\draw(0,0) circle (1.39);	
			\end{tikzpicture}
		}
		\hspace{10pt}
		\caption{\quad $d=1$}
		\label{tu1}
	\end{figure}
	
	As its high-dimension generalization, we have the following:
	
	\begin{example}\label{app3}
		Assume that $\alpha \in\mathbb{T}^d$ is rationally independent with $d>1$, and $\omega\in\mathbb{T}$.	  Consider the operator
		\begin{equation*}
			(H_{\alpha}(\lambda,\omega) u)( n)=\sum_{i=1}^{d}u(n-e_i)+\lambda e^{2\pi i(\omega+ \langle n,\alpha \rangle )} u(n) , \quad n \in \mathbb{Z}^d.
		\end{equation*}
		Then  $$\sigma({H}_{\alpha}(\lambda, \omega))=\begin{cases}
		\{z\in\mathbb{C}\big||z|\leq d\}, & |\lambda|\leq\lambda_1,\\
		\{z\in\mathbb{C}\big|G_d^{-1}(\log|\lambda|)\leq|z|\leq d\}, & \lambda_1<|\lambda|<d,\\
		\{z\in\mathbb{C}\big||z|= |\lambda|\},& |\lambda|\geq d,		
	\end{cases}
	$$
	for all $ \omega\in\mathbb{T}$, where $\lambda_1=e^{G_d(0)}$ and $G_d^{-1}(\cdot)$ is the inverse function of $G_d(|z|)$.
	\end{example}
	
	\begin{remark}
	The case $d>1$ is interesting, this gives another example of non-self-adjoint  quasi-periodic operator, which possesses non-simply-connected two-dimensional spectrum, as shown in Figure \ref{tu3} (B). 
	\end{remark}	
\begin{figure}[h!]
	\centering  
	
	\subfloat[$|\lambda|\leq$$\lambda_1$]{
		\begin{tikzpicture}[baseline,remember picture]
			\fill[gray!30] (0, 0) -- (1, 0) arc (0:180:1) -- (0, 0) -- (0, 0) arc (180:0:0);
			\fill[gray!30] (0, 0) -- (1, 0) arc (360:180:1) -- (0, 0) -- (0, 0) arc (180:360:0);	
			\draw (0, 0) circle (1);	
		\end{tikzpicture}
	}
	\hspace{10pt}  
	\subfloat[$\lambda_1$<$|\lambda|$<$d$]{
		\begin{tikzpicture}[baseline,remember picture]
			\fill[gray!30] (0, 0) -- (1, 0) arc (0:180:1) -- (-0, 0) -- (-0.3, 0) arc (180:0:0.3);
			\fill[gray!30] (0, 0) -- (1, 0) arc (360:180:1) -- (-0, 0) -- (-0.3, 0) arc (180:360:0.3);
			
			\draw (0, 0) circle (0.3);
			\draw (0, 0) circle (1);	
		\end{tikzpicture}
	}
	\subfloat[$|\lambda|\geq$$d$]{
		\begin{tikzpicture}[baseline,remember picture]
			\draw(0,0) circle (1.39);	
		\end{tikzpicture}
	}
	\hspace{10pt}
	\caption{\quad $d>1$}
	\label{tu3}
\end{figure}

The following example not only shows that the non-degeneracy condition \eqref{non-dege} in Corollary \ref{mfre} is necessary, but also   shows that the spectrum can be independent of the dimension $d$.
\begin{example}\label{two6}
		Assume that $\alpha \in\mathbb{T}^d$ is rationally independent and $\omega\in\mathbb{T}$.	  Consider the operator
	\begin{equation*}
		(H_{\alpha}(\lambda,\omega) u)( n)=u(n_1-1,\cdots,n_d-1)+\lambda e^{2\pi i(\omega+ \langle n,\alpha\rangle)} u(n) , \quad n=(n_1,\cdots,n_d) \in \mathbb{Z}^d.
	\end{equation*}Then
	\begin{equation*}
		\sigma({H}_{\alpha}(\lambda, \omega))=
		\begin{cases}
			\mathbb{S}^1, & |\lambda|<1,\\
			\mathbb{D}, & |\lambda|=1,\\
			\lambda\mathbb{S}^1,& |\lambda|>1,		
		\end{cases}
	\end{equation*}
	for all $\omega\in\mathbb{T}.$
\end{example}

For $d = 1$, two-dimensional spectrum occurs in  $P_{\lambda}$ and $\sigma({H}_{\alpha}(\lambda, \omega))=P_{\lambda}$ in the aforementioned examples. In the following, we give an example with $C_{|\lambda|} \neq \emptyset$.

 	\begin{example}\label{two5}
Assume that $\alpha \in\mathbb{T}\backslash\mathbb{Q}$ and $\omega\in\mathbb{T}$.	  Consider the operator
 	\begin{equation*}
 		(H_{\alpha}(\lambda,\omega) u)( n)=2u(n-2)+u(n-1)+2 e^{2\pi i(\omega+ n\alpha)} u(n) , \quad n \in \mathbb{Z}.
 	\end{equation*}
 	Then
 	$$\sigma({H}_{\alpha}(\lambda, \omega))=R(V)\cup {D}
 	$$
 	for all $ \omega\in\mathbb{T}$, where $D$ is the regime surround by $$\{e^{2\pi i\theta}+2e^{4\pi i\theta}|\theta\in[\frac{1}{2\pi}\arccos(-\frac{1}{4}),1-\frac{1}{2\pi}\arccos(-\frac{1}{4})]\}.$$ 
 \end{example}
 
 	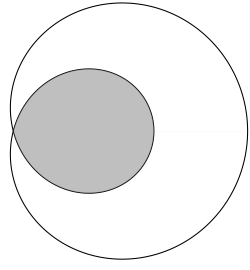
\begin{figure}[h!]
 	\centering  
 
 \begin{tikzpicture}
 	\begin{polaraxis}[
 		hide axis
 		]
 		\addplot[save path=\pathA, mark=none,domain=0:180,samples=600, black] 
 		{1+4*cos(x)};
 		\addplot[save path=\pathB,  mark=none,domain=180:360,samples=600, black] 
 		{1+4*cos(x)};
 		\clip[reuse path=\pathA];
 		\fill[gray,opacity=0.5,reuse path=\pathB];
 	\end{polaraxis}
 \end{tikzpicture}
    	\caption{ Two-dimensional spectrum of second-order operator }
    \label{tu6}
\end{figure}

As  mentioned before, we only need to assume that $V(\cdot)$ is H\"{o}lder continuous. In the following, we provide an explicit example  of H\"{o}lder continuous but not differentiable  potential such that its spectrum exhibits a real-complex phase transition.
	\begin{example}\label{non-analytic}
		Assume that $\alpha \in\mathbb{T}\backslash\mathbb{Q}$ and $\omega\in\mathbb{T}$.	  Consider the operator
		\begin{equation*}
			(H_{\alpha}(\lambda,\omega) u)( n)= \sum_{k\in\mathbb{Z}}v_ku(n-k)+\lambda e^{2\pi i(\omega+ n\alpha)} u(n) , \quad n \in \mathbb{Z},
		\end{equation*}    
	where $v_k$ is the Fourier coefficient of \begin{equation*}
			V(\theta)=
			\begin{cases}
				\theta, & 0\leq\theta<\frac{1}{2},\\
				1-\theta, &\frac{1}{2}\leq\theta<1.	
			\end{cases}
		\end{equation*}
		Then for all $ \omega\in\mathbb{T}$,
		$$\sigma({H}_{\alpha}(\lambda, \omega)=\begin{cases}
			[0,\frac{1}{2}], & |\lambda|\leq\lambda_2,\\
			\{z\in[0,\frac{1}{2}]\big|G_1(z)> \log|\lambda|\}\cup\{z\in\mathbb{C}\big|G_1(z)= \log|\lambda|\}, & \lambda_2<|\lambda|<\lambda_3,\\
			\{z\in\mathbb{C}\big|G_1(z)= \log|\lambda|\},& |\lambda|\geq\lambda_3,		
		\end{cases}
		$$
 where $\lambda_2=e^{\min_{z\in[0,{1}/{2}]} G_1(z)}$, $\lambda_3=e^{G_1(0)}$. Moreover, we can compute $G_1(z)$ exactly as follows
		\begin{equation*}
			G_1(z)=\left\{
			\begin{aligned}
				&2\Re(z)\log(|z|)-(\Re(z)-\frac{1}{2})\log(\Re(z)-\frac{1}{2})^2-1,\quad \Im(z)=0,\\
				&2\Re(z)\log(|z|)-(\Re(z)-\frac{1}{2})\log((\Re(z)-\frac{1}{2})^2+\Im(z))+\\&2\Im(z)(\arctan(\frac{\Im(z)}{2\Im(z)^2+2\Re(z)^2-\Re(z)}))-1, \quad \Im(z)\neq0.\\
			\end{aligned}
			\right.
		\end{equation*}
	\end{example}
			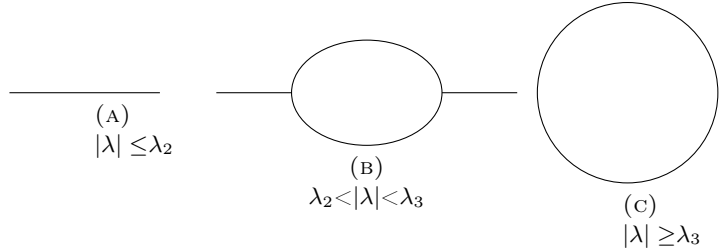
\begin{figure}[h!]
		\centering  
		\subfloat[$|\lambda|\leq$$\lambda_2$]{
			\begin{tikzpicture}[baseline,remember picture]
				\draw[-] (-1,0)--(1,0);
			\end{tikzpicture}
		}
		\hspace{10pt}  
		\subfloat[$\lambda_2$<$|\lambda|$<$\lambda_3$]{
			\begin{tikzpicture}[baseline,remember picture]
				\draw[-] (2,0)--(3,0);
				\draw[black] (4,0) ellipse (1 and 0.7);
				\draw[-] (5,0)--(6,0);
			\end{tikzpicture}
		}
		\subfloat[$|\lambda|\geq$$\lambda_3$]{
			\begin{tikzpicture}[baseline,remember picture]
				\draw(0,0) circle (1.2);	
			\end{tikzpicture}
		}
		\hspace{10pt}
		\caption{\quad Spectrum of non-analytic potential }
		\label{tu4}
	\end{figure}

\section{Preliminaries}
\subsection{Continued fraction expansion}
Let $\alpha \in\mathbb{T} \backslash \mathbb{Q}, a_0=0$, and let $\alpha_0=\alpha$. Inductively for $j \geq 1$,
$$
a_j=[\alpha_{j-1}^{-1}], \quad \alpha_j=\alpha_{j-1}^{-1}-a_j=\{\frac{1}{\alpha_{j-1}}\} .
$$
Let $p_0=0, p_1=1, q_0=1, q_1=a_1$, then we define inductively $p_j=a_j p_{j-1}+p_{j-2}, q_k=a_k q_{k-1}+q_{k-2}$. The sequence $(q_n)$ are the denominators of the best rational approximations of $\alpha$ and we have
$$
\forall 1 \leq k<q_n, \quad\|k \alpha\|_{\mathbb{T}} \geq\|q_{n-1} \alpha\|_{\mathbb{T}},
$$
and
$$
\|q_n \alpha\|_{\mathbb{T}} \leq \frac{1}{q_{n+1}},
$$
where
$$
\|x\|_{\mathbb{T}}=\inf _{p \in \mathbb{Z}}|x-p| .
$$
Let $$\beta(\alpha):=\limsup _{n \rightarrow \infty} \frac{\ln q_{n+1}}{q_n}.$$ For any $\gamma>0$, we have $\{\alpha|\beta(\alpha)=\gamma\}$ is dense in $\mathbb{T}$.

\subsection{Regular value Theorem}

We recall the  concept of regular value and the  basic regular value Theorem:

\begin{definition}[Regular value  \cite{Poschel}]\label{pos2}
	Let $f:A\rightarrow F$ be a continuous differentiable map from an open subset $A$ of a Banach space $E$ into  another Banach space ${F}$. A point $c$ in $\mathbb{F}$ is a regular value if for every point $x$ in the level set $M_c=\{x\in A: f(x)=c\}$ there exists a splitting $E=E_h\oplus E_v$, such that $d_xf|_{E_v}$ is a linear isomorphism between $E_v$ and $F$. In particular, if $E=\R^n$, $F=\R$, then $c$ is a regular value of $f$ if and only if $\nabla f\neq 0$ at every point $x$ in $M_c$.
\end{definition}

\begin{definition}[Real analytic submanifold \cite{Poschel}]
	Let $E$, $F$ be real Banach space, let $\C E$, $\C F$ be their complexifications, and let $U\subset E$ be open. A map $f:U\rightarrow F$ is real analytic on $U$, if for each point in $U$ there is a neighborhood $V\subset\C E$ and an analytic map $g:V\rightarrow \C F$, such that $f=g$ on $U\cap V.$ A subset $M$ of a Banach space $E$ is a real analytic submanifold of $E$, if for every point $x\in M$, there is a real analytic coordinate system $\phi:U\rightarrow V$ between an open neighborhood $U$ of $x$ in $E$ and an open subset $V$ of another Banach space $F$ and a splitting $F=F_h\oplus F_v$ such that $\phi (U\cap M)=V\cap F_h.$   
\end{definition}

\begin{theorem}[Regular value Theorem \cite{Poschel}]\label{pos1}
	Suppose $f:A\rightarrow F$ is a real analytic map from an open set of a Banach space $E$ into another Banach space $F$, If $c\in{F}$ is a regular value of $F$, then $M_c=\{x\in A: f(x)=c\}$ is a real analytic submanifold of E. Moreover, $T_xM_c=ker d_xf$ at every point $x$ in $M_c$. 
\end{theorem}

\subsection{Some general facts on Aubry duality}
Aubry duality \cite{AA} was first developed for the analysis of the almost Mathieu operator, but the approach can be extended to other potentials, even non-self-adjoint quasi-periodic operators.  	Previous proofs   of Aubry duality preserving the spectrum  rely on the self-adjointness of the operator \cite{AA,Jit,Sim2}. 
The aim of this section is to introduce the method for exploring the operator and its duality using decomposable operators, without requiring self-adjointness.

We consider the operator$$
H_{\lambda, \alpha}=\int_{\mathbb{T}}^{\oplus} H^V_{\alpha}(\lambda,\omega) \mathrm{d} \omega: \int_{\mathbb{T}}^{\oplus} \ell^2(\mathbb{Z}^d) \mathrm{d} \omega\rightarrow \int_{\mathbb{T}}^{\oplus} \ell^2(\mathbb{Z}^d) \mathrm{d} \omega
$$
given by
$$
(H_{\lambda, \alpha} \varphi)(\omega, n)=\sum_{k\in\mathbb{Z}^d}v_k\varphi(\omega, n-k)+\lambda e^{2\pi i(\omega+\langle n,\alpha\rangle)} \varphi(\omega, n) .
$$

The duality transform
$$
\mathcal{U}:\int_{\mathbb{T}}^{\oplus} \ell^2(\mathbb{Z}^d) \mathrm{d} \omega \rightarrow \int_{\mathbb{T}^d}^{\oplus}\ell^2(\mathbb{Z}) \mathrm{d} \theta
$$
 given by
$$
(\mathcal{U} \varphi)(\theta, n)=\sum_{m \in \mathbb{Z}^d} \int_{\mathbb{T}} e^{-2 \pi i\langle\theta+n \alpha, m\rangle} e^{-2 \pi i n \eta} \varphi(\eta, m) \mathrm{d} \eta 
$$
is unitary by  definition. Moreover, one can check that
$$(\widehat{H}_{\lambda, \alpha}\psi)(\theta,n):=\lambda\psi(\theta,n-1)+V(\theta+n\alpha)\psi(\theta,n)=(\mathcal{U}H_{\lambda, \alpha}\mathcal{U}^{-1}\psi)(\theta,n).$$

Consider the direct integral $H_{\lambda, \alpha}=\int_{\mathbb{T}}^{\oplus} H^V_{\alpha}(\lambda,\omega) \mathrm{d} \omega$ and $\widehat{H}_{\lambda, \alpha}=\int_{\mathbb{T}^d}^{\oplus} \widehat{H}^V_{\alpha}(\lambda, \theta) \mathrm{d} \theta$. Since $\mathcal{U}$ is an unitary operator, we have
$$\sigma(\int_{\mathbb{T}}^{\oplus} H^V_{\alpha}(\lambda,\omega) \mathrm{d} \omega)=\sigma(\int_{\mathbb{T}^d}^{\oplus} \widehat{H}^V_{\alpha}(\lambda, \theta) \mathrm{d} \theta).$$

Since the potential is quasi-periodic, i.e. the base dynamic is minimal, then  the spectrum of operator is independent of phase.
\begin{lemma}[\cite{lim}]\label{ind}
	There exist some $\Sigma_1,\Sigma_2\subset\mathbb{C}$ such that $\sigma(H^V_{\alpha}(\lambda,\omega))=\Sigma_1$ for all $\omega\in\mathbb{T}$ and $\sigma(\widehat{H}^V_{\alpha}(\lambda, \theta))=\Sigma_2$ for all $\theta\in\mathbb{T}^d.$
\end{lemma}

As a consequence, we have the following:

\begin{corollary}\label{rot}
	$\sigma(H^V_{\alpha}(\lambda,\omega))=\sigma(H^V_{\alpha}(|\lambda|,\omega))$ for all $\lambda\in\mathbb{C}$ and $\omega \in \mathbb{T}.$
\end{corollary}

Now we can prove  that the resolvent of the direct integral is equal to the resolvent of the original operator. 

\begin{proposition}\label{cor}
	Assume that $\alpha \in \mathbb{T}^d$ is rationally independent, $V\in C (\T^d,\C)$, $\lambda \in \mathbb{C}$, then we have
	\begin{equation}\label{equal}\rho(\int_{\mathbb{T}^d}^{\oplus} \widehat{H}^V_{\alpha}(\lambda, \theta) \mathrm{d} \theta)=\rho(\widehat{H}^V_{\alpha}(\lambda, \tilde{\theta})) \quad\text{for all }\tilde{\theta}\in\mathbb{T}^d,\end{equation} and
	$$\rho(\int_{\mathbb{T}}^{\oplus} {H}^V_{\alpha}(\lambda, \omega) \mathrm{d} \omega)=\rho({H}^V_{\alpha}(\lambda, \tilde{\omega})) \quad\text{ for all } \tilde{\omega}\in\mathbb{T}.$$
	Consequently, $$\sigma(H^V_{\alpha}(\lambda, \omega))=\sigma(\widehat{H}^V_{\alpha}(\lambda, \theta))$$ for all $\omega\in\mathbb{T},\theta\in\mathbb{T}^d.$
\end{proposition}

\begin{proof}
	As the argument is analogous, we only  prove  \eqref{equal}.	
	According to the general theory of  direct integrals \cite{Barry}, we have $$
	\begin{aligned}\label{direct}
		&\quad\{z|z\in\cap_{\theta}\rho(\widehat{H}^V_{\alpha}(\lambda, \theta)),\sup_{\theta}\|(z-\widehat{H}^V_{\alpha}(\lambda, \theta))^{-1}\|<\infty\}\\&\subset\rho(\int_{\mathbb{T}^d}^{\oplus} \widehat{H}^V_{\alpha}(\lambda, \theta) \mathrm{d} \theta)\subset\{z|z\in\cap_{\theta}\rho(\widehat{H}^V_{\alpha}(\lambda, \theta))\}.
	\end{aligned}
	$$
	For any $z\in\cap_{\theta}\rho(\widehat{H}^V_{\alpha}(\lambda, \theta))$
	and any $\theta \in\mathbb{T}^d $,    since $V\in C (\T^d,\C)$,    if  $\theta' \in\mathbb{T}^d$ such that $|\theta-\theta'|$ is sufficiently small, then 
	$$ |V(\theta)-V(\theta')| \|(z-\widehat{H}^V_{\alpha}(\lambda, \theta))\| \leq \frac{1}{2}.$$ It follows that
	$$
	\begin{aligned}
		&\quad \|(z-\widehat{H}^V_{\alpha}(\lambda, \theta'))^{-1}\|-\|(z-\widehat{H}^V_{\alpha}(\lambda, \theta))^{-1}\|\\	&\leq\|(z-\widehat{H}^V_{\alpha}(\lambda, \theta'))^{-1}-(z-\widehat{H}^V_{\alpha}(\lambda, \theta))^{-1}\|\\
		&=\|(z-\widehat{H}^V_{\alpha}(\lambda, \theta'))^{-1}(\widehat{H}^V_{\alpha}(\lambda, \theta')-\widehat{H}^V_{\alpha}(\lambda, \theta)(z-\widehat{H}^V_{\alpha}(\lambda, \theta))^{-1}\|\\
		&\leq \|(z-\widehat{H}^V_{\alpha}(\lambda, \theta'))^{-1}\|\|(z-\widehat{H}^V_{\alpha}(\lambda, \theta))^{-1}\|  |V(\theta)-V(\theta')|  \\
		&\leq\frac{1}{2}\|(z-\widehat{H}^V_{\alpha}(\lambda, \theta'))^{-1}\|.
	\end{aligned}
	$$
	Hence we have $$\|(z-\widehat{H}^V_{\alpha}(\lambda, \theta'))^{-1}\|\leq2{\|(z-\widehat{H}^V_{\alpha}(\lambda, \theta))^{-1}\|}.$$
	Since $\mathbb{T}^d$ is compact, by the standard compactness argument, we have $$\sup_{\theta}\|(z-\widehat{H}^V_{\alpha}(\lambda, \theta))^{-1}\|<\infty.$$
	It follows that $$\rho(\int_{\mathbb{T}^d}^{\oplus} \widehat{H}^V_{\alpha}(\lambda, \theta) \mathrm{d} \theta)=
	\{z|z\in\cap_{\theta}\rho(\widehat{H}^V_{\alpha}(\lambda, \theta))\}.$$
	According to Lemma \ref{ind}, we have $$\rho(\int_{\mathbb{T}^d}^{\oplus} \widehat{H}^V_{\alpha}(\lambda, \theta) \mathrm{d} \theta)=\rho(\widehat{H}_{\alpha}(\lambda, \tilde{\theta})) \quad\text{ for all }\tilde{\theta}\in\mathbb{T}^d,$$ the result follows.
\end{proof}

\section{Proof of  Theorem \ref{thm}}
In what follows, we will always assume $V(\cdot)$ is 
H\"{o}lder continuous, and $\log|z-V(\cdot)|\in L^1(\mathbb{T}^d)$ for all $z\in\mathbb{C}$.
By Corollary \ref{rot} and Proposition \ref{cor}, the proof of Theorem \ref{thm} reduces to the following two propositions.  

\begin{proposition}\label{res}
	For any  rationally independent $\alpha \in \mathbb{T}^d$, and any $\lambda>0$, 
	$$\sigma(\widehat{H}^V_{\alpha}(\lambda, \theta))\subset S_{\lambda}  \quad \text{for any} \quad  \theta\in\mathbb{T}^d.$$
\end{proposition}

Proposition \ref{res} was essentially proved  in \cite{Boca,Sarnak}, and we  give the proof in Appendix A just for completeness.

\begin{proposition}\label{res2}
	For any  rationally independent $\alpha \in \mathbb{T}^d$, and any $\lambda>0$, 
	$$S_{\lambda}\subset \sigma(\widehat{H}^V_{\alpha}(\lambda, \theta))\quad \text{for any}  \quad  \theta\in\mathbb{T}^d.$$
\end{proposition}

We are left to give the proof of Proposition \ref{res2}. Recall that $S_{\lambda}=P_{\lambda} \cup C_{\lambda},$ we will prove $P_{\lambda} \subset \sigma(\widehat{H}^V_{\alpha}(\lambda, \theta))$ and $C_{\lambda}\subset \sigma(\widehat{H}^V_{\alpha}(\lambda, \theta))$ separately. 

{\subsection{Proof of $P_{\lambda}\subset \sigma(\widehat{H}^V_{\alpha}(\lambda, \theta))$:}
Recall  $P_{\lambda}=\{z\in\mathbb{C}|G_d(z)=\log\lambda\}.$}

\begin{lemma}\label{0}
	For any  rationally independent $\alpha \in \mathbb{T}^d$,  and any $\lambda>0$, $P_{\lambda}\subset\sigma(\widehat{H}^V_{\alpha}(\lambda, \theta))$  for any $\theta\in \mathbb{T}^d$.
\end{lemma}
\begin{proof}
	Since $G_d(z)=\log\lambda$, we can choose $\tilde{\theta}$ such that
	\begin{equation}\label{zerol}
		\lim _{n \rightarrow \infty} \frac{1}{n}  \sum_{s=0}^{n-1} \log |z-V(\tilde{\theta}\pm s\alpha)|-\log\lambda 
		=G_d(z)-\log\lambda=0.
	\end{equation}
	Consider the  solution $\psi=(\dots,\psi_{-1},\psi_0,\psi_1,\dots)$ of $\widehat{H}^V_{\alpha}(\lambda, \tilde{\theta})\psi=z\psi$ with the initial value condition $\psi_0=1.$ 
	That is, for $n>0$
	$$\psi_n=\frac{\lambda^n}{ \prod_{j=1}^{n} (z-V({\theta}+j\alpha))},$$
	while 		$$\psi_{-n}=\frac{ \prod_{j=0}^{n-1}(z-V({\theta}-j\alpha))}{\lambda^n},$$
	consequently by \eqref{zerol}, we have
	\begin{equation}\label{grow}
		\lim\limits_{|n|\rightarrow\infty}\frac{\log|\psi_n|}{n}=0.
	\end{equation}  
	Define $$\psi^{n}=\chi_{[-n,n]}\psi, \quad\widetilde{\psi}^{n}=\frac{\psi^n}{\|\psi^n\|}.$$
	Then, we have $$(\widehat{H}^V_{\alpha}(\lambda, \theta)-z)\psi^n=-\lambda\psi_{-n-1}\delta_{-n}+\lambda\psi_{n}\delta_{n+1},$$
	where $\{\delta_n\}_{n\in\mathbb{Z}}$ is the usual orthonormal basis of $\ell^2(\mathbb{Z})$. 
	Hence
	$$\|(\widehat{H}^V_{\alpha}(\lambda, \theta)-z)\widetilde{\psi}^{n}\|^2=\frac{|\lambda\psi_{-n-1}|^2+|\lambda\psi_{n}|^2}{\|\psi^n\|^2}.$$
	
	Suppose that there is $\delta>0$ and $n_0>0$ such that for any $n\geq n_0$, we have $$||(\widehat{H}^V_{\alpha}(\lambda, \theta)-z)\widetilde{\psi}^{n}\|>\delta.$$
	This implies that
	$$\|\psi^{n+1}\|^2-\|\psi^{n-1}\|^2>\|\psi^n\|^2\frac{\delta}{\lambda^2}>\|\psi^{n-1}\|^2\frac{\delta}{\lambda^2},$$ that is
	$$\|\psi^{n+1}\|^2>\|\psi^{n-1}\|^2(1+\frac{\delta}{\lambda^2}).$$
	Consequently, we have
	\begin{equation}\label{grow-1}		
		\|\psi^{2n+n_0}\|^2\geq\|\psi^{n_0}\|^2(1+\frac{\delta}{\lambda^2})^n.
	\end{equation}

	On the other hand, noting
	$$\|\psi^{2n+n_0}\|^2\leq (4n+2n_0+1)(\max_{|k|\leq 2n+n_0}|\psi_k|)^2,$$ 
	combining with \eqref{grow-1}, this implies 
	$$\log((4n+2n_0+1)(\max_{|k|\leq 2n+n_0}|\psi_k|))^2\geq 2\log\|\psi^{n_0}\|+n\log(1+\frac{\delta}{\lambda^2}).$$	
It further implies
	$$\limsup_{n\rightarrow\infty} \frac{\log((4n+2n_0+1)(\max_{|k|\leq 2n+n_0}|\psi_k|))^2}{2n+n_0}\geq \frac{1}{2}\log(1+\frac{\delta}{\lambda^2}),$$ that is to say
	$$\limsup_{n\rightarrow\infty} \frac{\log(\max_{|k|\leq 2n+n_0}|\psi_k|)}{2n+n_0}\geq \frac{1}{4}\log(1+\frac{\delta}{\lambda^2}).$$
	This is impossible.
	In fact, suppose there exists $\{n_m\}$, s.t. $$\lim_{m\rightarrow\infty} \frac{\log(\max_{|k|\leq 2n_m+n_0}|\psi_k|)}{2n_m+n_0}> 0.$$
	Choose $|\psi_{n_m^0}|=\max_{|k|\leq 2n_m+n_0}|\psi_k|,$ then
	$$\lim_{m\rightarrow\infty}\frac{\log|\psi_{n_m^0}|}{2|n_m^0|}=\lim_{m\rightarrow\infty}\frac{\log|\psi_{n_m^0}|}{2|n_m^0|+n_0}\geq\lim_{m\rightarrow\infty} \frac{\log(\max_{|k|\leq 2n_m+n_0}|\psi_k|)}{2n_m+n_0}> 0,$$
	which  contradicts with \eqref{grow}. Once we have this, then result follows immediately from Weyl's criterion.\end{proof}

\begin{remark}
	Here, $L(z):=G_d(z)-\log\lambda$ can be seen as Lyapunov exponent of the  singular cocycle $(\alpha,
	\frac{ z-V(\theta)}{\lambda}).$ One can compare Lemma \ref{0} with the quasi-periodic Schr\"odinger operator
	\begin{equation}\label{schro}
		(H_{V,\alpha,\theta} u)(n)= u(n+1)+u(n-1) +V (n\alpha + \theta) u(n)=z u(n),
	\end{equation}
	and  also denote its Lyapunov exponent by $L(z)$, it just means   if $L(z)=0$,  then $z$ belongs to the spectrum.

\end{remark}

\subsection{ Proof of  $C_{\lambda}\subset \sigma(\widehat{H}^V_{\alpha}(\lambda, \theta))$:}
It remains to show that $$C_{\lambda}=\{z \in \mathbb{C}: G_d(z)>\log\lambda\} \cap R(V)\subset\sigma(\widehat{H}^V_{\alpha}(\lambda, \theta)).$$ We first recall a fundamental fact in functional analysis.

\begin{lemma}\cite{A}\label{point}
	A point $z$ is an approximate eigenvalue of a bounded operator  $A$ if there is a sequence $\phi_n$ such that $\|\phi_n\|=1$ and $\|(A-z)\phi_n\|\rightarrow 0.$		
	If $z\in\sigma(A)$, either $z$ is an approximate eigenvalue of $A$ or $\bar{z}$ is an  eigenvalue of $A^*$, where $A^*$ is the adjoint operator of $A$.
\end{lemma}

Let $T(\alpha,\theta)=(t_{i j})_{i, j \in \mathbb{Z}}$ be a one-parameter infinite-dimensional matrix (or an operator on $\ell^2(\Z)$). Our first aim is to show the continuity principle of  its spectrum 
$\mathcal{S}(T(\alpha)):=\cup_{\theta\in\mathbb{T}^d}\sigma(T(\alpha,
\theta))$:

\begin{proposition}\label{jian}(Weak continuity of the spectrum)
	Let $V(\cdot)\in C(\mathbb{T}^d,\mathbb{C})$ be H\"{o}lder continuous and $T(\alpha,\theta)=(t_{i j})_{i, j \in \mathbb{Z}}$, where $$t_{j,j}=V(\theta_1+j\alpha_1,\dots,\theta_d+j\alpha_d), \qquad t_{j, j\pm k}=a_{\pm k},$$ $k=1,2,\dots m$ and $ t_{j, j\pm k}=0$ for all $k>m$.  If  $\alpha^{(n)}\rightarrow\alpha$ and $z\in \mathcal{S}(T(\alpha^{(n)}))$ for all $n$, then   $z\in \mathcal{S}(T(\alpha))$.
\end{proposition}
\begin{proof}
	
	Let $\varepsilon_n>0$ with $\varepsilon_n\rightarrow 0$. We choose $\alpha^{(n)}$ such that $|\alpha^{(n)}-\alpha|=\varepsilon_n^{3/t}$ and $N_n=\varepsilon_n^{-2/t},$ where $t$ is the H\"{o}lder exponent of $V(\cdot)$, respectively.
	We note that  $z\in \mathcal{S}(T(\alpha^{(n)})$, therefore, there exists $\theta^{(n)}\in\mathbb{T}^d$, such that $z\in\sigma(T(\alpha^{(n)},\theta^{(n)}))$, which implies that  either $z$ is an approximate eigenvalue of $T(\alpha^{(n)},\theta^{(n)})$ or $\bar{z}$ is an eigenvalue of $T(\alpha^{(n)},\theta^{(n)})^*$ by Lemma \ref{point}.

	We  require a more precise estimate of the approximation sequence:
	\begin{lemma}\label{app}
		Let $T=(t_{i j})_{i, j \in \mathbb{Z}}$ be an infinite-dimensional matrix with $|t_{j, j\pm k}| \leq M$, $k=1,2,\dots m$ and $ t_{j, j\pm k}=0$ for all $k>m$. Suppose that 0 is an approximate eigenvalue of $T$. Let $N$ be a positive integer and let $\varepsilon>0$. Then there exist an integer $k$, constant $C(m,M)$ and a non-zero vector $\phi\in\ell^2(\mathbb{Z})$ with
		$$
		\phi=(\ldots, 0,0, \phi_{k+1}, \phi_{k+2}, \ldots, \phi_{k+N}, 0,0, \ldots) 
		$$
		such that
		$$
		\|T \phi\| \leq(\frac{C}{\sqrt{N}}+\varepsilon)\|\phi\|.
		$$
	\end{lemma}
	\begin{remark}	The proof of Lemma \ref{app} is similar to \cite{Ch}, where they proved the case of Jacobi operators. We make slight modifications and give the proof in Appendix B just for completeness. 
	\end{remark}

	Once we have this, if $z$ is  an approximate eigenvalue of $T(\alpha^{(n)},\theta^{(n)})$, by Lemma \ref{app}, there exists an integer $k$ and a non-zero vector
	$$
	\phi_n=\left(\ldots, 0,0, \phi_{k+1}, \phi_{k+2}, \ldots, \phi_{k+N_n}, 0,0, \ldots\right) \in \ell^2(\mathbb{Z})
	$$
	such that
	$$
	\|(T(\alpha^{(n)}, \theta^{(n)})-z) \phi_n\| \leq (\frac{C}{\sqrt{N_n}}+\varepsilon_n)\|\phi_n\|.
	$$
	We choose $\tilde{\theta}^{(n)}+(k+\frac{N_n+1}{2})(\alpha-\alpha^{(n)})=\theta^{(n)}$. With this choice, we have	
	$$
	(T(\alpha, \tilde{\theta}^{(n)})-T(\alpha^{(n)}, \theta^{(n)})) \phi_n=\left(\ldots, 0,0, \gamma_{k+1} \phi_{k+1}, \ldots, \gamma_{k+N_n} \phi_{k+N_n}, 0,0, \ldots\right),
	$$
	where
	$$
	\gamma_s= V(\tilde{\theta}^{(n)}+s\alpha)-V(\theta^{(n)}+s\alpha^{(n)}), \quad k+1 \leq s \leq k+N_n.
	$$
	And one has an estimate
	$$
	\begin{aligned}
		|\gamma_s|&\leq |s(\alpha-\alpha^{(n)})+\tilde{\theta}^{(n)}-\theta^{(n)}|^t\\&=|s-k-\frac{N_n+1}{2}|^t|\alpha^{(n)}-\alpha|^t\leq (N_n-1)^t|\alpha^{(n)}-\alpha|^t=\varepsilon_n.
	\end{aligned}$$
	Therefore, we have
	
	\begin{eqnarray}
		\nonumber			\|T(\alpha, \tilde{\theta}^{(n)}-z)\phi_n\|&\leq& \|(T(\alpha^{(n)}, \theta^{(n)})-z)\phi_n\|+\|(T(\alpha, \tilde{\theta}^{(n)}-T(\alpha^{(n)}, \theta^{(n)}))\phi_n\|\\
		\nonumber		&\leq& (\frac{C}{\sqrt{N_n}}+\varepsilon_n)\|\phi_n\|+\varepsilon_n\|\phi_n\|\\
		\label{tal}		& \leq& (C+2)\varepsilon_n\|\phi_n\|.
	\end{eqnarray}

	If  $\bar{z}$ is an  eigenvalue of $T(\alpha^{(n)},\theta^{(n)})^*$, similarly, we have
	\begin{equation}\label{talc}
		\|(T(\alpha, \tilde{\theta}^{(n)})^*-\bar{z}) \phi_n\| \leq (C+2)\varepsilon_n\|\phi_n\|.
	\end{equation}
	Therefore, there exist $\varepsilon_n,\tilde{\theta}^{(n)},\phi_n$ such that for any $n$, either \eqref{tal} or \eqref{talc} holds.
	
	Since $\mathbb{T}^d$ is  compact, there  exists an accumulation point $\tilde{\theta}$, such that $|\tilde{\theta}^{(n)}-\tilde{\theta}|=c_n\rightarrow 0$ along some subsequence. If \eqref{tal} holds infinitely many times, then
	$$\begin{aligned}
		\|(T(\alpha, \tilde{\theta})-z) \phi_n\|&\leq\|(T(\alpha, \tilde{\theta}^{(n)})-z) \phi_n\|+\|(T(\alpha, \tilde{\theta}^{(n)})-T(\alpha, \tilde{\theta})) \phi_n\|\\
		& \leq(C+2)\varepsilon_n\|\phi_n\|+c_n^t |\phi_n\|.
	\end{aligned}
	$$
	This means
	$$\|(T(\alpha, \tilde{\theta})-z) \frac{\phi_n}{\|\phi_n\|}\|\rightarrow 0,\quad n\rightarrow\infty.$$
	According to  Weyl's criterion, $z\in\sigma(T(\alpha, \tilde{\theta}))\subset\mathcal{S}(T(\alpha))$.
	If \eqref{talc} holds infinitely many times, then $\bar{z}\in\sigma(T(\alpha, \tilde{\theta})^*)$, we  also have $z\in\sigma(T(\alpha, \tilde{\theta}))\subset\mathcal{S}(T(\alpha))$.
\end{proof}

With the aforementioned preparations, we can prove our conclusion using rational approximation. We recall the following lemma about the periodic operator.
\begin{lemma}\label{periodic}\cite{KK}
	Let $(Ju)(n)=au(n+1)+b(n)u(n)+cu(n-1)$ is an operator on $\ell^2(\mathbb{Z}),$ where $a,c\in\mathbb{C}$. Suppose $b(\cdot):\mathbb{Z}\rightarrow\mathbb{C}$ satisifies $b(n+N)=b(n).$ Then
	$$\sigma(J)=\cup_{\phi\in[0,2\pi]}\{z|\det(z-J(\phi))=0\},$$
	where $J(\phi)= 
	\begin{pmatrix}
		b_{0} & a &  & & ce^{i\phi}\\
		c & b_{1} & a & & \\
		& \ddots & \ddots & \ddots& \\
		& & c & b_{N-2}& a\\
		ae^{-i\phi} &  & & c & b_{N-1}\\
	\end{pmatrix}$.
\end{lemma}

We will first state the  result in the one-frequency case, which is more instructive. 

\begin{proposition}\label{gordon}
	Let $\lambda>0$, $V(\cdot)$ is $t$-H\"{o}lder continuous. 
	Then for any  $$t \beta(\alpha)> G_1(z)-\log\lambda>0$$ 
	with $z\in R(V)$, we have	 $z\in\sigma(\widehat{H}^V_{\alpha}(\lambda, \theta))$  for  any $\theta\in \mathbb{T}^d$.
\end{proposition}
\begin{proof}
	We first need the following elementary observation:
	
	\begin{lemma}
		For any $\varepsilon>0$, there exist $K(\varepsilon)$, $N(\varepsilon)$ such that
		$$\sup_{\theta\in\mathbb{T}}\prod_{j=0}^{k-1}|z-V({\theta}+j{\alpha})|\leq e^{k(G_1(z)+\varepsilon)},$$ 
		$$\sup_{\theta\in\mathbb{T}}\prod_{j=0}^{k-1}|z-V({\theta}+j{\frac{p_n}{q_n}})|\leq e^{k(G_1(z)+\varepsilon)}$$ for any $n>N$, $k>K$. Moreover, we have
		\begin{equation}\label{appro}\prod_{j=0}^{q_n-1}|z-V({\theta}+j\alpha)|-\prod_{j=0}^{q_n-1}|z-V({\theta}+j{\frac{p_n}{q_n}})| \leq e^{(\log\lambda+\varepsilon)q_n}. \end{equation}
	\end{lemma}
	\begin{proof}
	For any $\varepsilon>0$, we can choose $\delta(\varepsilon)$, such that $$\tilde{G}_1(z)=\int_{\mathbb{T}}\log\bigg(|z-V({\theta})|+\delta\bigg)\mathrm{d}\theta<G_1(z)+\frac{\varepsilon}{4}.$$
	{By uniquely ergodic theorem, $$\lim_{k \rightarrow \infty}\frac{1}{k}\sum_{n=0}^{k-1}\log\bigg(|z-V({\theta}+n\alpha)|+\delta\bigg)=\tilde{G}_1(z)$$ converges uniformly in $\theta$. }
		Then  there exists $K(\varepsilon)>0$, such that for any $k\geq K$,  $$\sup_{\theta\in\mathbb{T}}\prod_{j=0}^{k-1}|z-V({\theta}+j{\alpha})|\leq\sup_{\theta\in\mathbb{T}}\prod_{j=0}^{k-1}\bigg(|z-V({\theta}+j{\alpha})|+\delta\bigg)\leq e^{k(\tilde{G}_1(z)+\frac{\varepsilon}{4})}\leq e^{k(G_1(z)+\frac{\varepsilon}{2})}.$$ 
		Fixed $l\geq K$, then
		$$\sup_{\theta\in\mathbb{T}}\bigg|\frac{1}{l}\sum_{j=0}^{l-1}\log\bigg(|z-V({\theta}+j{\frac{p_n}{q_n}})|+\delta\bigg)-\frac{1}{l}\sum_{j=0}^{l-1}\log\bigg(|z-V({\theta}+j{\alpha})|+\delta\bigg)\bigg|\leq \frac{\varepsilon}{2}$$
		for $q_n$ sufficiently large.
		This means for $q_n$ sufficiently large,
		\begin{equation}\label{up}
			\sup_{\theta\in\mathbb{T}}\prod_{j=0}^{l-1}|z-V({\theta}+j{\frac{p_n}{q_n}})|\leq	\sup_{\theta\in\mathbb{T}}\prod_{j=0}^{l-1}\bigg(|z-V({\theta}+j{\frac{p_n}{q_n}})|+\delta\bigg)\leq e^{l({G}(z)+\varepsilon)}.
		\end{equation}
		Thus, there exists $N(\varepsilon)$, such that \eqref{up}
		holds for $K\leq l\leq2K-1$ if $n>N$.
		Since any $k\geq K$ can be written as a sum of integers $l_i$ satisfying $K\leq l_i\leq 2K-1$, it follows that \eqref{up} is true for  all $k\geq K$.
		
		As a consequence, by the telescoping argument, we have
		$$\begin{aligned}
			&\quad\prod_{j=0}^{q_n-1}|z-V({\theta}+j\alpha)|-\prod_{j=0}^{q_n-1}|z-V({\theta}+j{\frac{p_n}{q_n}})|\\&\leq\bigg|\prod_{j=0}^{q_n-1}(z-V({\theta}+j\alpha))-\prod_{j=0}^{q_n-1}(z-V({\theta}+j\frac{p_n}{q_n}))\bigg|\\&\leq \bigg|\bigg(\sum_{l=0}^{K}+\sum_{l=K+1}^{q_n-1-K}+\sum_{l=q_n-K+1}^{q_n-1}\bigg)\prod_{j=l+1}^{q_n-1}(z-V({\theta}+j\alpha))\prod_{j=0}^{l-1}(z-V({\theta}+j\frac{p_n}{q_n})) |\alpha-\frac{p_n}{q_n}|^t\bigg|  \\
			&\leq q_n^2e^{(G_1(z)+\varepsilon) (q_n-1)}|\alpha-\frac{p_n}{q_n}|^t\\
			&\leq e^{(G_1(z)+\varepsilon) q_n}e^{-tq_n(\beta(\alpha)-\varepsilon)}\\
			&\leq e^{(\log\lambda+\varepsilon)q_n}.
		\end{aligned}
		$$
		We thus finish the proof. 
	\end{proof}

	Once we have this, we can apply periodic approximation, and obtain the following:
	
	\begin{lemma}\label{equ-p}
		If further $z\in R(V)$, then for sufficiently large $n$,  there exists ${\theta}^{(n)}$ such that 
		\begin{equation}\label{per}\prod_{j=0}^{q_n-1}|z-V({\theta}^{(n)}+j{\frac{p_n}{q_n}})|=\lambda^{q_n}.\end{equation}
	\end{lemma}
	\begin{proof}
		Since $\alpha\in\mathbb{T}\verb|\|\mathbb{Q}$, we can find $\overline{\theta}\in\mathbb{T}$ such that $$\lim_{k \rightarrow \infty}\frac{1}{k}\sum_{n=0}^{k-1}\log|z-V(\overline{\theta}+n\alpha)|-\log\lambda=G_1(z)-\log\lambda>0.$$
		Then \eqref{appro} implies 
		\begin{equation*}
			\begin{aligned}
				\prod_{j=0}^{q_n-1}|z-V(\overline{\theta}+j{\frac{p_n}{q_n}})|&\geq\quad\prod_{j=0}^{q_n-1}|z-V(\overline{\theta}+j\alpha)|-e^{(\log\lambda+\varepsilon)q_n}\\
				&\geq e^{(G_1(z)-\epsilon) q_n}-e^{(\log\lambda+\varepsilon)q_n} >  \lambda^{q_n}.
			\end{aligned}
		\end{equation*}
		
		Now under the assumption  $z\in R(V)$, we can  find $\underline{\theta}$ such that
		$$\quad\prod_{j=0}^{q_n-1}|z-V(\underline{\theta}+j{\frac{p_n}{q_n}})|=0.$$
		By the continuity, the above inequality implies 
		there exists ${\theta}^{(n)}$, such that
		$$\quad\prod_{j=0}^{q_n-1}|z-V({\theta}^{(n)}+j{\frac{p_n}{q_n}})|=\lambda^{q_n}.$$
	\end{proof}

	Note \eqref{per} means there exists $\phi_n\in[0,2\pi]$, such that
	\begin{equation}\label{deter1}
		\prod_{j=0}^{q_n-1}(z-V({\theta}^{(n)}+j{\frac{p_n}{q_n}}))=(-1)^{q_n+1}\lambda^{q_n}e^{i\phi_n}.
	\end{equation}
	Consider the operator $\widehat{H}^V_{\frac{p_n}{q_n}}(\lambda, {\theta}^{(n)})$, it is a periodic operator with period $q_n$.
	Applying Lemma \ref{periodic} with $N=q_n$, $a=0$, $b_j=V({\theta}^{(n)}+j{\frac{p_n}{q_n}})$, $c=\lambda$,  we see that $$\det(z-J(\phi_n))=\prod_{j=0}^{q_n-1}(z-V({\theta}^{(n)}+j{\frac{p_n}{q_n}}))-(-1)^{q_n+1}\lambda^{q_n}e^{i\phi_n}=0$$   by \eqref{deter1}, we can conclude that
	$z\in\sigma(\widehat{H}^V_{\frac{p_n}{q_n}}(\lambda, {\theta}^{(n)})).$
	By the weak continuity of the spectrum (Proposition \ref{jian}), it follows that
	$z\in\sigma(\widehat{H}^V_{\alpha}(\lambda, \tilde{\theta}))$ for some $\tilde{\theta}.$
	Since $\alpha\in\mathbb{T}\verb|\|\mathbb{Q}$, $\sigma(\widehat{H}^V_{\alpha}(\lambda, {\theta}))$ is independent of $\theta$. Thus, $z\in\sigma(\widehat{H}^V_{\alpha}(\lambda, {\theta}))$ for any $\theta\in\mathbb{T}$.
\end{proof}

\begin{remark}\label{go-cla}
	One can also compare this result with \eqref{schro}, if the potential
	$V(\cdot)$ is $t$-H\"{o}lder continuous, 
	sharp Gordon's Lemma  \cite{AYZ} shows that if $$t\beta(\alpha)> L(z),$$ then $z$ is not an eigenvalue of $H_{V,\alpha,\theta}$. 
\end{remark}

\subsection{Proof of Proposition \ref{res2}:}

Now we finish the proof of Proposition \ref{res2}:

\begin{proof}[Case $d=1$:] By  Lemma \ref{0}, we only need to prove $C_{\lambda}\subset\sigma(\widehat{H}^V_{\alpha}(\lambda, \theta))$. 
	For any $z\in C_{\lambda}$, if $t \beta(\alpha)>G_1(z)-\log\lambda$, 
	by Proposition \ref{gordon},  $z\in \sigma(\widehat{H}^V_{\alpha}(\lambda, \theta))$. Since 
	$\{\alpha| t \beta(\alpha)>G_1(z)-\log\lambda\}$ is dense in $\mathbb{T},$ again by  the weak continuity  (Proposition \ref{jian}), one conclude
	$C_{\lambda}\subset\sigma(\widehat{H}^V_{\alpha}(\lambda, \theta))$ for any $\alpha\in\mathbb{T}\verb|\|\mathbb{Q}$. \end{proof}

\begin{proof}[Case $d>1$:]
	If $d>1$, the proof is similar to the  $d=1$. Here, we only provide a brief outline of the proof.
	\begin{lemma}\label{dio}
		Let $\lambda>0$, $z\in C_{\lambda}$. Then for any $\alpha_1,\dots,\alpha_{d-1}\in\mathbb{Q}$,  there exists $\gamma(z,\alpha_1,\dots,\alpha_{d-1})$, such that if $\alpha_d\in\mathbb{T}\verb|\|\mathbb{Q}$ with $\beta(\alpha_d)>\gamma(z,\alpha_1,\dots,\alpha_{d-1})$, then $z\in\sigma(\widehat{H}^V_{\alpha}(\lambda, \theta))$  for $\alpha=(\alpha_1,\dots,\alpha_{d-1},\alpha_d)$ and any $\theta\in \mathbb{T}^d$.
	\end{lemma}
	\begin{proof}
		For simplicity, we only present the proof for $d=2$, the other case is analogous.
		Let $z\in C_{\lambda}$. Fix any  $\alpha_1=\frac{p}{q}\in\mathbb{Q}$.
		Choose $\gamma(z,\alpha_1)=\frac{(G_2(z)-\log\lambda )q}{t}$,
		where $t$ is the H\"{o}lder exponent of $V(\cdot).$
		For any $\alpha=(\alpha_1,\alpha_2)$, choose $\alpha_2\in\mathbb{T}\verb|\|\mathbb{Q}$ with $\beta(\alpha_2)>\gamma(z,\alpha_1)$.
		Consider the continued fraction expansion of $\alpha_2\in\mathbb{T}\verb|\|\mathbb{Q}$, without loss of generality, we assume $$\beta(\alpha_2)=\lim_{n\rightarrow\infty}\frac{\ln q_{n+1}}{q_n},$$	and denote $\alpha^{(n)}=(\frac{p}{q},\frac{p_n}{q_n})$.
		
		By  telescoping arguments, similar as Lemma \ref{equ-p},   for sufficiently large $n$,  there exists ${\theta}^{(n)}$ such that  
		$$\prod_{j=0}^{qq_n-1}|z-V({\theta}^{(n)}+j{\alpha^{(n)}})| =\lambda^{qq_n},$$
		This means there exists $\phi_n\in[0,2\pi]$, such that
		\begin{equation*}\label{deter}
			\prod_{j=0}^{qq_n-1}(z-V({\theta}^{(n)}+j{\alpha^{(n)}}))=(-1)^{qq_n+1}\lambda^{qq_n}e^{i\phi_n},
		\end{equation*}
		consequently $z\in\sigma(\widehat{H}^V_{\alpha^{(n)}}(\lambda, {\theta}^{(n)})).$ 
		Then the result follows from  Proposition \ref{jian}.	\end{proof}

	Once we have this, define the set $$ A(z,\alpha_1,\dots,\alpha_{d-1})=\{\alpha=(\alpha_1,\dots,\alpha_{d-1},\alpha_d)|\alpha_d\in\mathbb{T}\verb|\|\mathbb{Q},\beta(\alpha_d)>\gamma(z,\alpha_1,\dots,\alpha_{d-1})\}.$$
	For any $\alpha=(\alpha_1,\dots,\alpha_{d-1},\alpha_d)\in\mathbb{T}^d$ and any $\epsilon>0$, there exists $\tilde{\alpha}_1,\dots,\tilde{\alpha}_{d-1}\in\mathbb{Q}$, such that $\sum_{i=1}^{d-1}|\alpha_i-\tilde{\alpha}_{i}|<\frac{\epsilon}{2}.$ Since  $\{\alpha|\beta(\alpha)>\gamma(z,\tilde{\alpha}_1,\dots,\tilde{\alpha}_{d-1})\}$ is dense in $\mathbb{T},$ we can find $\tilde{\alpha}_d$ with 
	$$\beta(\tilde{\alpha}_d)>\gamma(z,\tilde{\alpha}_1,\dots,\tilde{\alpha}_{d-1})$$ such that 	 $|\alpha_d-\tilde{\alpha}_d|<\frac{\epsilon}{2}.$
	Therefore $$\cup_{\alpha_1,\dots,\alpha_{d-1}\in\mathbb{Q}^{d-1}}A(z,\alpha_1,\dots,\alpha_{d-1})$$ is a dense set in $\mathbb{T}^d.$ 
	By Lemma  \ref{dio} and Proposition \ref{jian}, we have $C_{\lambda}\subset\sigma(\widehat{H}^V_{\alpha}(\lambda, \theta))$ for any rationally independent $\alpha$. This, combined with Lemma \ref{0}, gives the conclusion.
\end{proof}

\section{Proof of main results}
\subsection{Basic properties on $G_d(z)$}
For the convenience of proof,  we sometimes write $z=x+iy$ and $G_d(x,y):=G_d(z).$
\begin{lemma}\label{cont}
	Let $V(\cdot)\in C^{\omega}(\mathbb{T}^d,\mathbb{C})$, then $G_d(z)$ is a continuous function, and $G_d(z)$ can attain its minimum in $\C$. Moreover, if $V(\cdot)\in C^{\omega}(\mathbb{T}^d,\mathbb{R})$, 
	\begin{enumerate}
		\item\label{p3} $G_d(x,y)=G_d(x,-y)$;
		\item\label{p}$ \frac{ \partial}{\partial y}G_d(x,y)>0$ for any $x\in \mathbb{R}$, $y>0$; 
		\item\label{p1} $ \frac{ \partial}{\partial x}G_d(x,y)>0$ for any $y\in \mathbb{R}$, $x>\max_{\theta\in\mathbb{T}^d}V(\theta)$; 
		\item \label{p2}  	$ \frac{ \partial}{\partial x}G_d(x,y)<0$ for any $y\in \mathbb{R}$, $x<\min_{\theta\in\mathbb{T}^d}V(\theta)$. 
	\end{enumerate}
	Consequently, $G_d(z)$ can only attain its minimum in $R(V)$.
\end{lemma}
\begin{proof}
	To establish the continuity of $G_d(z)$, we recall the following theorem:
	\begin{theorem}\cite{Powell}\label{powell}
		Suppose $\omega = (\omega_1,\cdots,\omega_d) \in \T^d.$ Let $(A, \omega)$ be an analytic quasi-periodic $M(2,\C)$-cocycle.
		Suppose, moreover, that $\det(A) \not\equiv 0.$ Then
		$L(A,\omega)$ is continuous in $A$ for any $\omega \in \T^d.$ 
	\end{theorem}
	Note that $G_d(z)$ is Lyapunov exponent of the quasi-periodic cocycle $$(\alpha,\begin{pmatrix}
		z-V(\theta)&0\\0&z-V(\theta)
	\end{pmatrix}).$$ 
By virtue of Theorem \ref{powell}, we conclude that  $G_d(z)$ is continuous on $\mathbb{C}$. Because $\lim\limits_{z\rightarrow\infty}G_d(z)=\infty$, there is an $M>0$ such that $G_d(z)>1$ when $|z|>M$. On the closed disk $D_M:=\{z\big||z|\leq M\}$, $G_d(z)$ attains a minimum value $m$ since it is continuous on $D_M$. Consequently, $\min\{1,m\}$ is the minimum value of $G_d(z)$ over $\C$.
	
	If $V(\cdot)$ is a real analytic function,    then one can compute
\begin{equation}\label{continuous}
	G_d(x,y)=\int_{\mathbb{T}^d}\log|x+iy-V(\theta)|\mathrm{d}\theta=\frac{1}{2}\int_{\mathbb{T}^d}\log((x-V(\theta))^2+y^2)\mathrm{d}\theta.
\end{equation}  
From \eqref{continuous}, it is obvious that $G_d(x,y)=G_d(x,-y)$.  For the other conclusions, we only prove Lemma \ref{cont}\eqref{p}, since the proof of  the others are similar.
Fix any $x\in\mathbb{R}$, for any $y>0$, $\log((x-V(\theta))^2+y^2)$ and $\frac{2y}{(x-V(\theta))^2+y^2}$ are continuous with $(\theta,y)$ in $\T^d\times[\frac{y}{2},\frac{3}{2}y]$. Therefore, we can interchange the order of integration and differentiation.
Hence, we have $\frac{ \partial}{\partial y}G_d(x,y)=\int_{\mathbb{T}^d}\frac{2y}{(x-V(\theta))^2+y^2}\mathrm{d}\theta>0.$

According to \eqref{p3}\eqref{p}, we have $G_d(x,y)>G_d(x,0)$ for any $x\in \R$, $y\neq0.$ According to \eqref{p1}\eqref{p2}, we have  $G_d(x,0)>G_d(\max_{\theta\in\mathbb{T}^d}V(\theta),0)$ for any $x>\max_{\theta\in\mathbb{T}^d}V(\theta)$, and $G_d(x,0)>G_d(\min_{\theta\in\mathbb{T}^d}V(\theta),0)$ for any $x<\min_{\theta\in\mathbb{T}^d}V(\theta)$. Therefore, $G_d(z)$ can only attain its minimum in $R(V)$.
\end{proof}
$G_d(z)$ is not only a continuous function, but a subharmonic function.

\begin{lemma}\label{subharmonic}
	$G_d(z)$ is  subharmonic on $\mathbb{C}$, and harmonic on $\mathbb{C}\backslash R(V).$
\end{lemma}
\begin{proof}
	Let us first recall a basic fact in potential theory.
	\begin{lemma}\cite{Ransford}\label{potential}
		Let $(\Omega,\mu)$ be a measure space with $\mu(\Omega)<\infty$, let $U$ be an open subset of $\mathbb{C}$, and let $v:U\times\Omega\rightarrow[-\infty,\infty)$ be a function such that:
		\begin{itemize}
			\item $v$ is measurable on $U\times\Omega;$
			\item $z\mapsto v(z,\omega)$ is subharmonic on $U$ for each $\omega\in \Omega;$
			\item $z\mapsto\sup_{\omega\in\Omega}v(z,\omega)$ is locally bounded above on $U$.
		\end{itemize}
		Then $u(z):=\int_{\Omega}v(z,\omega)\mathrm{d}\mu(\omega)$ is subharmonic on $U$.
	\end{lemma}
	Set $\Omega=\mathbb{T}^d$, $\mathrm{d}\mu=\mathrm{d}\theta$, $U=\mathbb{C}$. Applying lemma \ref{potential} with $v(z,\theta)=\log|z-V(\theta)|$. We see that $G_d(z)$ is subharmonic on $\mathbb{C}$. Applying the same theorem with $v(z,\theta)=-\log|z-V(\theta)|$ on $\mathbb{C}\backslash R(V)\times\mathbb{T}^d$, we also find that $G_d(z)$ is superharmonic on $\mathbb{C}\backslash R(V)$, and hence harmonic there.
\end{proof}

\subsection{Proof of main results}

By Corollary \ref{rot}, in the following, we just assume $\lambda \in \R$ with $\lambda>0$.  

\subsubsection{Proof of Theorem \ref{app1}}
Theorem \ref{app1} is a consequence of Theorem \ref{thm}.\qed

\subsubsection{Proof of Corollary \ref{mfre}}
Suppose $\tilde{\theta}=(\tilde{\theta}_1,\cdots,\tilde{\theta}_d)$ and 
$$\text{rank}(\frac{\partial(\Re V)}{\partial \theta_i}(\tilde{\theta}),\frac{\partial(\Im V)}{\partial \theta_i}(\tilde{\theta}))_{1\leq i\leq d}=2.$$ Then there exist ${\theta}_l,{\theta}_k$ such that $$\begin{vmatrix}
	\frac{\partial(\Re V)}{\partial \theta_l}(\tilde{\theta}) & \frac{\partial(\Im V)}{\partial \theta_l}(\tilde{\theta})\\
	\frac{\partial(\Re V)}{\partial \theta_k}(\tilde{\theta}) & \frac{\partial(\Im V)}{\partial \theta_k}(\tilde{\theta})\\
\end{vmatrix}\neq0.$$
By inverse function Theorem, $$(\theta_l,\theta_k)\mapsto(\Re V(\tilde{\theta}_1,\cdots,\theta_l,\cdots,\theta_k,\cdots,\tilde{\theta}_d),\Im V(\tilde{\theta}_1,\cdots,\theta_l,\cdots,\theta_k,\cdots,\tilde{\theta}_d))$$ is homeomorphism in the neighborhood of $(\tilde{\theta}_l,\tilde{\theta}_k)$, and $R(V)$ has interior.

On  the other hand, by Lemma \ref{cont},  $\min_{z\in \C}G_d(z)$ exists. Therefore
if $|\lambda|<e^{\min_{z\in \C}G_d(z)}$, then $G_d(z) > \log \lambda$.  Theorem \ref{thm} imply that
$\sigma({H}^V_{\alpha}(\lambda, \omega))=C_{\lambda}=R(V)$, the result then follows. \qed

\subsubsection{Proof of Theorem \ref{two4}}
Without loss of generality, we can assume $v_0=0$ by shifting the spectrum, and distinguish the proof into two cases:\\

\smallskip
\textbf{Case 1: }$l>0$.
For any $0<|z|<|v_m|-\sum_{k=l}^{m-1}|v_k|$, consider $f_1(z,\textbf{z})=z-v_m\textbf{z}^m$, $f_2(\textbf{z})=-\sum_{k=l}^{m-1}v_k\textbf{z}^k$. When 
$|\textbf{z}|=1$,
$$|f_1(z,\textbf{z})|\geq |v_m|-|z|>\sum_{k=l}^{m-1}|v_k|\geq|f_2(\textbf{z})|.$$
According to Rouche's theorem, $f(z,\textbf{z}):= z-\sum_{k=l}^{m}v_k\textbf{z}^k$ has $m$ roots $\textbf{z}_1,\textbf{z}_2,\cdots,\textbf{z}_m$, where $|\textbf{z}_i|<1$.

By Jensen's formula and Vieta's theorem, 
$$G_d(z)=\log|z|-\sum_{i=1}^{m}\log|\mathbf{z}_i|=\log|v_m|.$$
If $|\lambda|=|v_m|$, according to Theorem \ref{thm}, we have $$\{z\big||z|\leq|v_m|-\sum_{k=l}^{m-1}|v_k|\}\subset\sigma({H}^V_{\alpha}(\lambda, \theta))$$ since the  spectrum is a closed set.\\

\smallskip
\textbf{Case 2:} $l<0$.
We assume $v_l\neq0$. For any $|z|<|v_m|-\sum_{k=l}^{m-1}|v_k|$, consider $f_1(\textbf{z})=z\textbf{z}^{-l}-v_m\textbf{z}^{m-l}$, $f_2(z,\textbf{z})=-\sum_{k=0}^{m-l-1}v_{k+l}\textbf{z}^k$. When 
$|\textbf{z}|=1$,
$$|f_1(\textbf{z})|\geq|v_m|- |z|>\sum_{k=l}^{m-1}|v_k|\geq|f_2(z,\textbf{z})|.$$
According to Rouche's theorem, $$f(z,\textbf{z}):=z\textbf{z}^{-l}-\sum_{k=0}^{m-l}v_{k+l}\textbf{z}^k$$ has $m-l$ roots $\textbf{z}_1,\textbf{z}_2,\cdots,\textbf{z}_{m-l}$, where $|\textbf{z}_i|<1$.

By Jensen's formula and Vieta's theorem, 
$$G_d(z)=\log|v_l|-\sum_{i=1}^{m-l}\log|z_i|=\log|v_m|.$$ 
If $|\lambda|=|v_m|$, according to Theorem \ref{thm}, we have $$\{z\big||z|\leq|v_m|-\sum_{k=l}^{m-1}|v_k|\}\subset\sigma({H}^V_{\alpha}(\lambda, \theta))$$ since the spectrum is a closed set. \qed

\subsubsection{Proof of Corollary \ref{shape}}
As $R(V)=\partial D$ for a bounded open set $D$, the subset $C_{\lambda}\subset R(V)$ lacks an interior. Therefore $\mathring{\sigma}({H}^V_{\alpha}(\lambda, \omega))\subset P_{\lambda}$.
The proof is divided into two parts:\\

\smallskip
\textbf{Part 1} 	We begin by proving that $\mathring{\sigma}({H}^V_{\alpha}(\lambda, \omega))\subset D$.
If $\mathring{\sigma}({H}^V_{\alpha}(\lambda, \omega))\cap(\mathbb{C}\backslash{D})\neq\emptyset$. Since $\mathring{\sigma}({H}^V_{\alpha}(\lambda, \omega))$ is  open, there exists an open set $$B\subset\mathring{\sigma}({H}^V_{\alpha}(\lambda, \omega))\subset P_{\lambda}$$ such that $B\subset\mathbb{C}\backslash\overline{D}$. Lemma \ref{subharmonic} guarantees that $G_d(z)$ is harmonic on $\mathbb{C}\backslash\overline{D}$, by uniqueness theorem of harmonic function, $G_d(z)\equiv \log\lambda$ on $\mathbb{C}\backslash\overline{D}$. However, this contradicts the fact that $\lim\limits_{z\rightarrow\infty}G_d(z)=\infty$.\\

\smallskip
\textbf{Part 2}	Since $\mathring{\sigma}({H}^V_{\alpha}(\lambda, \omega))\subset D$, without loss of generality (by relabelling $D_j$),   their exists $l$ with $1\leq l\leq k$, such that $\mathring{\sigma}({H}^V_{\alpha}(\lambda, \omega))\subset\cup_{j=1}^l D_{j}$, where $D_j$ are the connected components of $D$.  Then their exists open subset $B_{j}\subset D_{j}$ such that $G_d(z)\equiv\log\lambda$ on $B_{j}$. Since $G_d(z)$ is harmonic on $\cup_{j=1}^{l}D_{j}$ by Lemma \ref{subharmonic},  we 
have $G_d(z)\equiv\log\lambda$ on $D_{j}$ by uniqueness theorem of harmonic function. According to Theorem \ref{thm}, $D_{j}\subset\sigma({H}^V_{\alpha}(\lambda, \omega)).$
Since $D_{j}$ is an open set, we also have $D_{j}\subset\mathring{\sigma}({H}^V_{\alpha}(\lambda, \omega)).$ Therefore $\cup_{j=1}^l D_{j}\subset\mathring{\sigma}({H}^V_{\alpha}(\lambda, \omega)).$
Consequently, $\mathring{\sigma}({H}^V_{\alpha}(\lambda, \omega))=\cup_{j=1}^l D_{j}.$
\qed

\subsubsection{Proof of Corollary \ref{pt}}
According to Lemma \ref{cont}, $G_d(z)$ is continuous. Therefore, $$\lambda_0(V)=e^{\min_{z\in R(V)}G_d(z)}, \qquad \lambda_1(V)=e^{\max_{z\in R(V)}G_d(z)}$$ are well-defined. 

(1)	
Furthermore according to Lemma \ref{cont}, $G_d(z)$ can only attain its minimum in $R(V)$. Therefore, if $|\lambda|<\lambda_0(V)$,  $$G_d(z)>\min_{z\in\C} G_d(z)=\min_{z\in R(V)} G_d(z)>\log\lambda.$$ Theorem \ref{thm} thus imply that $$\sigma({H}^V_{\alpha}(\lambda, \omega))=C_{\lambda}=R(V)=[\min\limits_{\theta\in\mathbb{T}^d}v(\theta),\max\limits_{\theta\in\mathbb{T}^d}v(\theta)].$$ 

If $\lambda=\lambda_0(V)$, Lemma \ref{cont} guarantees that $G_d(z)$ can only attain minimum on $R(V)$. As a result, 
$P_{\lambda}=\{z\in R(V)\big|G_d(z) \text{ is a minimum}\}$. Furthermore, $G_d(z)>\log\lambda$ in $\mathbb{C}\backslash P_{\lambda}$, and this implies that $$C_{\lambda}\cup P_{\lambda}=R(V)=[\min\limits_{\theta\in\mathbb{T}^d}v(\theta),\max\limits_{\theta\in\mathbb{T}^d}v(\theta)].$$

(2)
If $\lambda_0(V)<\lambda$, there exist $x_1\in R(V)$, such that $G_d(x_1,0)<\log \lambda$. Since $ \frac{ \partial}{\partial y}G_d(x,y)>0$ for any $x\in \mathbb{R}$, $y>0$ by Lemma \ref{cont}, and $\lim_{z\rightarrow \infty}G_d(z)=\infty.$  Thus there exists $y>0$, such that $G_d(x_1,y)=\log \lambda$. Therefore $$P_{\lambda}\cap(\mathbb{C}\verb|\|\mathbb{R})\neq\emptyset.$$

On the other hand, if $\lambda\leq\lambda_1(V)$, there exist $x_2\in R(V)$, such that $G_d(x_2,0)\geq \log\lambda$. Theorem \ref{thm} imply that $$\sigma({H}^V_{\alpha}(\lambda, \omega)) \cap\R\neq\emptyset.$$
Consequently, $\sigma(H^V_{\alpha}(\lambda,\omega)) $ has both  complex spectrum and real spectrum. 

(3) If $\lambda>\lambda_1(V)$, $G_d(z)<\log\lambda$ in $R(V)$ for $
|\lambda|>\lambda_1$. Therefore $M_0:=\{(x,y)\in\mathbb{R}^2 : G_d(x,y)-\log\lambda=0\}=\{(x,y)\in\mathbb{R}^2 \backslash R(V) : G_d(x,y)-\log\lambda=0\}$. 
By Lemma $\ref{cont}$, we have 
$\nabla G_d(x,y)=(\frac{\partial}{\partial x}G_d(x,y),\frac{\partial}{\partial y}G_d(x,y))\neq0$ for any $(x,y)\in\mathbb{R}^2\backslash R(V)$. Therefore $0$ is a regular value by Definition \ref{pos2}. Since $G_d(x,y)-\log\lambda$ is harmonic in $\mathbb{R}^2\backslash R(V)$ by Lemma \ref{subharmonic}, it is automatic real analytic. Hence, we have 
$$\sigma({H}^V_{\alpha}(\lambda, \omega))=P_{\lambda}=M_0$$
is a real analytic curve by Theorem \ref{pos1}.\qed

\subsubsection{Proof of Corollary \ref{twod}:}
	By Jensen's formula,

\begin{align}\label{jen2}
	G_1(z)&=\int_{\T}\log|z-e^{-g}e^{2 \pi i\theta}-e^{g}e^{-2\pi i\theta}|\mathrm{d}\theta  \notag \\&=\int_{\T}\log|ze^{2 \pi i\theta}-e^{g}-e^{-g}e^{4\pi i\theta}|\mathrm{d}\theta \notag\\
	&=\log e^{\Re g}-\min\{\log|\mathbf{z}_1|,0\}-\min\{\log|\mathbf{z}_2|,0\},		
\end{align}
where $z_1$, $z_2$ are roots of $e^{-g}\mathbf{z}^2-z\mathbf{z} +e^{g}=0$.
For any $z\in\mathbb{C}$, we can write $$z=\eta e^{-g}e^{2\pi i\theta}+\frac{1}{\eta} e^{g}e^{-2\pi i\theta},\eta\in(0,e^{\Re g}],\theta\in[0,1).$$
We can compute $\mathbf{z}_1=\eta e^{2\pi i\theta}$, $\mathbf{z}_2=\frac{1}{\eta}e^ge^{2\pi i\theta}$ and $|\mathbf{z}_2|=\frac{e^{\Re g}}{\eta}\leq 1$. By \eqref{jen2}, we have
\begin{equation}\label{jen3}
	G_1(z)=
	\begin{cases}
		\log e^{\Re g},& \eta\in[1,{e^{\Re g}}],\\
		\log e^{\Re g}-\log\eta, & \eta<1.\\
	\end{cases}
\end{equation}
Then Corollary \ref{twod} is a consequence of  Theorem \ref{thm}.\qed

\subsubsection{Proof of Corollary \ref{twod1}:}

	We need the following basic observation:

\begin{lemma}\label{App4} 
	If $V(\theta)=\cos2\pi\theta_1+e^{-g}e^{2 \pi i\theta_2}+e^{g}e^{-2\pi i\theta_2}$, 
	$G_2(z)=\int_{\mathbb{T}^2} \log |z-V(\theta)|\mathrm{d}\theta,$ then  we have
	\begin{enumerate}
		\item\label{g11} $G_2(z)\equiv g$ in $\overline{\mathrm{Conv}(R(V))\backslash R(V)}$;
		\item\label{g22} $G_2(z)>g$ in $\C\backslash \overline{\{{\mathrm{Conv}}(R(V))\backslash R(V)\}}$;
		\item\label{g33} $\nabla G_2(z)\neq 0$ for $|z|> 2\|V\|_{\infty}$.
	\end{enumerate}	
\end{lemma}
\begin{proof}

	We have
		\begin{align}\label{jen4}
		G_2(z)&=\int_{\T^2}\log|z-2\cos2\pi\theta_1-e^{-g}e^{2 \pi i\theta_2}-e^{g}e^{-2\pi i\theta_2}|\mathrm{d}\theta_1\mathrm{d}\theta_2  \notag \\&=\int_{\T}G_1(z-2\cos2\pi\theta_1)\mathrm{d}\theta_1,	
	\end{align}
		where $G_1(z)=\int_{\T}\log|z-e^{-g}e^{2 \pi i\theta_2}-e^{g}e^{-2\pi i\theta_2}|\mathrm{d}\theta.$

	\eqref{g11} We denote $$A:=\{\eta e^{-g}e^{2\pi i\theta}+\frac{1}{\eta} e^{g}e^{-2\pi i\theta}|\eta\in[1,{e^{g}}],\theta\in[0,1]\},$$
		i.e. the convex hull of elliptic $\{ e^{-g}e^{2\pi i\theta}+ e^{g}e^{-2\pi i\theta}|\theta\in[0,1]\}.$ 
	It is  easy to check that 
	\begin{equation}\label{obs}
	\overline{\mathrm{Conv}(R(V))\backslash R(V)}=(A+2)\cap(A-2).
	\end{equation}
which implies that 
$$ \overline{\mathrm{Conv}(R(V))\backslash R(V)}+[-2,2]\subset A.$$
			Once we have this,  it follows  from \eqref{jen3} 
			that $G_1(z-2\cos2\pi\theta_1)=g$ for all $\theta_1\in\T$ and any $z\in	\overline{\mathrm{Conv}(R(V))\backslash R(V)}$. 
	Hence we have $G_2(z)\equiv g$ in $\overline{\mathrm{Conv}(R(V))\backslash R(V)}$ by \eqref{jen4}.

		\eqref{g22}
		If $z\in \C\backslash \overline{\{{\mathrm{Conv}}(R(V))\backslash R(V)\}} $, 
		then as  another consequence of \eqref{obs}, we have
	   $$B:=\{\theta|z-2\cos2\pi\theta\notin A\}\neq\emptyset.$$
It follows that	\begin{align*}
			G_2(z)&=\int_{B}G_1(z-2\cos2\pi\theta_1)\mathrm{d}\theta_1+\int_{B^c}G_1(z-2\cos2\pi\theta_1)\mathrm{d}\theta_1\\
			&= \int_{B}G_1(z-2\cos2\pi\theta_1)\mathrm{d}\theta_1+\int_{B^c}g\mathrm{d}\theta_1\\
			&> \text{mes}(B)g+\text{mes}(B^c)g\\
			&=g,
		\end{align*}
	where the third inequality we use \eqref{jen3}.
	
			\eqref{g33}	If $|z|> 2\|V\|_{\infty}$, 
	suppose $|\Im z|\leq\|V\|_{\infty}$, then $|\Re z|>\|V\|_{\infty}$ by our assumption.
	Therefore,
	 $$|\Re z-\Re V(\theta)|\geq |\Re z|-|\Re V(\theta)|\geq|\Re z|-\|V\|_{\infty}>0.$$
	It follows that
	$$\nabla G_2(z)=\big(\int_{\mathbb{T}^2}\frac{\Re z-\Re V(\theta)}{|z-V(\theta)|^2}\mathrm{d}\theta,\int_{\mathbb{T}^2}\frac{\Im z-\Im V(\theta)}{|z-V(\theta)|^2}\mathrm{d}\theta\big)\neq 0.$$
		If $|\Im z|>\|V\|_{\infty}$, we  have
		$$|\Im z-\Im V(\theta)|\geq |\Im z|-|\Im V(\theta)|\geq|\Im z|-\|V\|_{\infty}>0.$$
		It follows that
		 	$$\nabla G_2(z)\neq 0.$$
\end{proof}

	Once we have this, we can finish the  proof.
\\(1) If $\lambda<e^g$, according to Lemma \ref{App4},
	$$G_2(z)>\min_{z\in\C} G_2(z)=g>\log\lambda.$$ Theorem \ref{thm} thus imply that $\sigma({H}_{\alpha}(\lambda, \omega))=R(V)=E.$
\\(2)	If $\lambda=e^g$, Lemma \ref{App4} and
	Theorem \ref{thm}  imply that $\sigma({H}_{\alpha}(\lambda, \omega))=\mathrm{Conv}(E).$\\
(3)	If $\lambda>e^{\max_{|z|\leq 2\|V\|}G_2(z)}$,
$$\max_{z\in R(V)}G_2(z)<\max_{|z|\leq 2\|V\|}G_2(z)<\log\lambda.$$
  Therefore, $$\sigma({H}_{\alpha}(\lambda, \omega))=P_\lambda=\{z|G_2(z)=\log\lambda\}=\{|z|>2\|V\|_{\infty}\big|G_2(z)=\log\lambda\}.$$
By Lemma $\ref{App4}$, we have 
$\nabla G_2(z)\neq0$ for any $|z|>2\|V\|_{\infty}$. Therefore $0$ is a regular value by Definition \ref{pos2}. Since $G_2(z)-\log\lambda$ is harmonic in $\mathbb{C}\backslash R(V)$ by Lemma \ref{subharmonic}, it is automatic real analytic. Hence, we have $\sigma({H}_{\alpha}(\lambda, \omega))$
	is a real analytic curve by Theorem \ref{pos1}.
	
\qed

\section{Proofs for the explicit examples}

To give the proof of explicit examples, the key is to explicitly calculate or estimate  $G_d(z)$.

\begin{proof}[Proof of Example \ref{app2}:]
	
	By Jensen's formula, one can compute 
	\begin{equation}\label{jensen}
		G_1(z)=\int_{\mathbb{T}}\log|z- e^{2\pi i\theta}|\mathrm{d}\theta=
		\begin{cases}
			0,& |z|<1,\\
			\log|z|, & |z|\geq1.\\
		\end{cases}
	\end{equation}
	Then the result follows from Theorem \ref{thm}.
\end{proof}

\begin{proof}[Proof of Example \ref{app3}:]
	We need the following basic observation:
	
		\begin{lemma}\label{App} 
		If $V(\theta)=\sum\limits_{j=1}^{d}e^{2\pi i\theta_j}$ with $d\geq1$, 
		$G_d(z)=\int_{\mathbb{T}^d} \log |z-\sum_{j=1}^d e^{2 \pi i \theta_j}|\mathrm{d}\theta,$ then
		
		\begin{enumerate}
			\item\label{g1} $G_d(z) $ is a radial and   monotone increasing convex function w.r.t. $|z|$;
			\item\label{g2} {$G_d(z)$ is strictly monotone increasing w.r.t. $|z|$ in $|z|\geq1$; }
			\item\label{g3} $G_d(z)=\log|z|$ for $|z|\geq d$.
		\end{enumerate}	
	\end{lemma}
	\begin{proof}
		
		\eqref{g1}	
		It is obvious that	$G_d(z)$ is a radial function w.r.t. $|z|$, we can combine this with the fact that $G_d(z)$ is a subharmonic function on $\C$  by Lemma \ref{subharmonic}. Consequently, 
		$G_d(z)$ is increasing and convex w.r.t. $|z|$ \cite{Ransford}.
		
		\eqref{g2}	We prove this statement by induction.	If $d=1$, \eqref{jensen} implies that $G_1(|z_1|)>G_1(|z_2|)$ for any $|z_1|>|z_2|\geq1$. If $d>1$, we assume that  $G_d(|z_1|)>G_d(|z_2|)$ for any $|z_1|>|z_2|\geq1$. Then for any $|z_1|>|z_2|\geq 1$, 
		
		\begin{align*}
			G_{d+1}(|z_1|)&=\int_{\mathbb{T}^{d+1}}\log\big||z_1|-\sum_{j=1}^{d+1} e^{2\pi i\theta_{j}}\big|\mathrm{d}\theta \\
			&=\int_\mathbb{T}G_d(|z_1|-e^{2\pi i\theta_{d+1}})) \mathrm{d}\theta_{d+1}\\
			&=\frac{1}{2}\int_\mathbb{T}G_d\bigg(|z_1|^2+1-2|z_1|\cos2\pi \theta_{d+1}\bigg) \mathrm{d}\theta_{d+1},
		\end{align*}
		where the third equality we use $G_d(z)$ is a radial function by Lemma \ref{App}\eqref{g1}.
		The function $f_\theta(x) = x^2 + 1 - 2x\cos2\pi\theta$ is strictly monotone increasing in  $[1,\infty)$. Therefore, whenever $|z_1|>|z_2| \geq 1$, we can conclude that $f_\theta(|z_1|)>f_\theta(|z_2|)$. By Lemma \ref{App}\eqref{g1}, we can see that  $G_d(\cdot)$ is monotone increasing function w.r.t. $|z|$. This implies that $$G_d(f_\theta(|z_1|))\geq G_d(f_\theta(|z_2|)),$$ i.e. 
		\begin{equation}\label{mo1}G_d\bigg(|z_1|^2+1-2|z_1|\cos2\pi \theta\bigg)\geq G_d\bigg(|z_2|^2+1-2|z_2|\cos2\pi \theta\bigg), \qquad \forall \theta \in\T. \end{equation}
		
		In particular, $$f_\frac{1}{4}(|z_1|)>f_\frac{1}{4}(|z_2|)\geq f_\frac{1}{4}(1)=2.$$	
		Therefore, their exists $\delta>0$, such that $f_\theta(|z_1|)>f_\theta(|z_2|)>1$ in $\theta\in(\frac{1}{4}-\delta,\frac{1}{4}+\delta)$. 	
		Our induction assumption implies that $G_d(\cdot)$ is strict monotone increasing function w.r.t. $|z|$ in $[1,\infty)$. Therefore, we have
		$$G_d(f_\theta(|z_1|))> G_d(f_\theta(|z_2|)),$$ and it follows that 
		\begin{equation}\label{mo2} G_d\bigg(|z_1|^2+1-2|z_1|\cos2\pi \theta\bigg)> G_d\bigg(|z_2|^2+1-2|z_2|\cos2\pi \theta\bigg), \qquad \theta\in  (\frac{1}{4}-\delta,\frac{1}{4}+\delta). 
		\end{equation} 	
		
		By \eqref{mo1} and \eqref{mo2}, we have 
		\begin{align*}
			G_{d+1}(|z_1|)
			&=\frac{1}{2}\int_{\frac{1}{4}-\delta}^{\frac{1}{4}+\delta}G_d\bigg(|z_1|^2+1-2|z_1|\cos2\pi \theta_{d+1}\bigg) \mathrm{d}\theta_{d+1}\\
			&+\frac{1}{2}\bigg(\int_{0}^{\frac{1}{4}-\delta}+\int_{\frac{1}{4}+\delta}^1\bigg)G_d\bigg(|z_1|^2+1-2|z_1|\cos2\pi \theta_{d+1}\bigg) \mathrm{d}\theta_{d+1}\\
			&>\frac{1}{2}\int_\mathbb{T}(G_d\bigg(|z_2|^2+1-2|z_2|\cos2\pi \theta_{d+1}\bigg) \mathrm{d}\theta_{d+1}\\
			&=G_{d+1}(|z_2|).
		\end{align*}
		i.e., $G_{d+1}(\cdot)$ is strict monotone increasing function w.r.t. $|z|$ in $[1,\infty)$.
		
		\eqref{g3} 
		Since $|z|\geq d$, we have $\bigg|z-\sum_{j=m}^{d} e^{2\pi i\theta_{j}}\bigg|\geq 1$ for any $2\leq m\leq d$. Therefore $$G_1(z-\sum_{j=m}^{d} e^{2\pi i\theta_{j}})=\log\big|z-\sum_{j=m}^{d} e^{2\pi i\theta_{j}}\big|$$ by \eqref{jensen}.
		It follows that
		\begin{eqnarray*}
			{G_d(z)}=\int_{\mathbb{T}^{d}}\log\big|z-\sum_{j=1}^{d} e^{2\pi i\theta_{j}}\big| \mathrm{d}\theta_{1}\cdots\theta_{d} 
			&=&\int_{\mathbb{T}^{d-1}}G_1(z-\sum_{j=2}^{d} e^{2\pi i\theta_{j}}) \mathrm{d}\theta_{2}\cdots\theta_{d}\\
			&=&\int_{\mathbb{T}^{d-1}}\log\big|z-\sum_{j=2}^{d} e^{2\pi i\theta_{j}}\big| \mathrm{d}\theta_{2}\cdots\theta_{d}.
		\end{eqnarray*}	
		By induction, we have 
		$$G_d(z)=\int_{\mathbb{T}}\log\big|z-e^{2\pi i\theta_{d}}\big| \mathrm{d}\theta_{d}=\log|z|.
		$$

	\end{proof}
	
	Once we have this, we can finish the  whole proof.\\
	(1) Based on Lemma \ref{App}\eqref{g1}, we know that $G_d(z)$ is   monotone increasing w.r.t. $|z|$. Therefore, we can conclude that  $\min_{z\in\C}G_d(z)=G_d(0)$. 
	
	Hence, if $|\lambda|<e^{G_d(0)}$,  $$G_d(z)>\min_{z\in\C} G_d(z)=G_d(0)>\log\lambda.$$ Theorem \ref{thm} thus imply that $$\sigma({H}_{\alpha}(\lambda, \omega))=C_{\lambda}=R(V)=\{z\in\mathbb{C}\big||z|\leq d\}.$$ 
	
	If $\lambda=e^{G_d(0)}$, Lemma \ref{App} guarantees that $G_d(z)$ can only attain minimum on $\{z\in\mathbb{C}\big||z|\leq 1\}$. As a result, 
	$P_{\lambda}=\{z\in R(V)\big|G_d(z) \text{ is a minimum}\}$. Furthermore, $G_d(z)>\log\lambda$ in $\mathbb{C}\backslash P_{\lambda}$ by Lemma \ref{App}\eqref{g2}, and this implies that $$C_{\lambda}\cup P_{\lambda}=R(V)=\{z\in\mathbb{C}\big||z|\leq d\}.$$ \\
		(2)
	We observe that if $\lambda>e^{G_d(0)}$,  their exists unique $r>0$ such that $G_d(r)=\log\lambda$ since $G_d(|z|)$ is a monotone increasing convex function  by Lemma \ref{App}\eqref{g1}. Hence $G_d^{-1}(\cdot)$ is well defined in $(e^{G_d(0)},\infty)$. 
	As a result, 
	$$P_{\lambda}=\{z\in\mathbb{C}\big|G_d(z)=\log\lambda\}=\{z\in\mathbb{C}\big||z|=G_d^{-1}(\log\lambda)\}.$$
	If $|z|<G_d^{-1}(\log\lambda)$, $G_d(z)<\log\lambda$  by Lemma \ref{App}\eqref{g1}.
	If $|z|>G_d^{-1}(\log\lambda)$, $G_d(z)>\log\lambda$ also by Lemma \ref{App}\eqref{g1}.
	Therefore, if $e^{G_d(0)}<\lambda<e^{G_d(d)}=d$,
	$$C_{\lambda}=\{z\in\mathbb{C}\big |G_d^{-1}(\log\lambda)<|z|\leq d\}.$$
	Theorem \ref{thm} thus imply that $$\sigma({H}_{\alpha}(\lambda, \omega))=P_{\lambda}\cup C_{\lambda}=\{z\in\mathbb{C}\big||G_d^{-1}(\log\lambda)\leq|z|\leq d\}.$$\\
	(3) As established by the  Lemma \ref{App}\eqref{g1}, $G_d(z)$ is a radial and monotone increasing function with respect to $|z|$. As a result, according to  Lemma \ref{App}\eqref{g3}, $\max_{z\in R(V)}G_d(z)=G_d(d)=\log d$.

	If $\lambda\geq d$, $G_d(z)<\log d\leq\log\lambda$ in $ \{z\in\mathbb{C}\big||z|<d\}$  by Lemma \ref{App}\eqref{g2}. And $G_d(z)=\log |z|$ in $ \{z\in\mathbb{C}\big||z|\geq d\}$  by Lemma \ref{App}\eqref{g3}.  Therefore $$\sigma({H}_{\alpha}(\lambda, \omega))=P_{\lambda}=\{z\in\mathbb{C}\big||z|= \lambda\}.$$

\end{proof}

\begin{proof}[Proof of Example \ref{two6}:] In this case, one can check that $V(\theta)=e^{\sum_{j=1}^{d}2\pi i\theta_j}$, then
	using variable substitution, we have
	\begin{equation*}
		G_d(z)=\int_{\mathbb{T}^d}\log|z-e^{\sum_{j=1}^{d}2\pi i\theta_j}|\mathrm{d}\theta=	\int_{\mathbb{T}}\log|z-e^{2\pi i t}|\mathrm{d}t=
		\begin{cases}
			0, &|z|<1,\\
			\log|z|, & |z|\geq1.\\
		\end{cases}
	\end{equation*}
	Then Example \ref{two6} is a consequence of  Theorem \ref{thm}.\end{proof}

\begin{proof}[Proof of Example \ref{two5}:]
	By Jensen's formula,
	\begin{equation}\label{jen1}
		G_1(z)=\int_{\T}\log|z-2e^{4\pi i\theta}-e^{2\pi i\theta}|\mathrm{d}\theta=\log|z|-\min\{\log|\mathbf{z}_1|,0\}-\min\{\log|\mathbf{z}_2|,0\}
	\end{equation}
	for $z\neq 0$, where $\textbf{z}_1,\textbf{z}_2$ are roots of $2\mathbf{z}^2+\mathbf{z}-z=0$, i.e.  $\mathbf{z}_1=\frac{-1+\sqrt{1+8z}}{4},\mathbf{z}_2=\frac{-1-\sqrt{1+8z}}{4}$.  
	
	(1) For any $0<|z|<1$, as $|\sqrt{1+8z}|\leq3$, which implies $|\mathbf{z}_1|\leq1, |\mathbf{z}_2|\leq1$.
	On the other hand, $\textbf{z}_1\textbf{z}_2=-\frac{z}{2}$, it follows that $G_1(z)=\log 2$ by \eqref{jen1}.
	By Theorem \ref{thm}, we have $\{z\big|0<|z|<1\}\subset\sigma({H}_{\alpha}(\lambda, \theta))$. 
	{Therefore, we have $D\subset\sigma({H}_{\alpha}(\lambda, \theta))$ by applying Corollary \ref{shape}.}

	(2) If $z=2e^{4\pi i\theta}+e^{2\pi i\theta}$, we can easily compute $$\textbf{z}_1=e^{2 \pi i\theta}, \textbf{z}_2=-\frac{1}{2}-e^{2\pi i\theta}.$$
	If $|\textbf{z}_2|\leq1$,
	then by \eqref{jen1}, we obtain $$G_1(z)=\log|z|-\log|\mathbf{z}_2|=\log|1+2e^{2\pi i\theta}|-\log|\mathbf{z}_2|=\log 2.$$
	If $|\textbf{z}_2|>1$, then by the same equation
	\eqref{jen1}, we have $$G_1(z)=\log|z|=\log|1+2e^{2\pi i\theta}|>\log 2.$$
	Therefore $R(V)\subset\sigma({H}_{\alpha}(\lambda, \theta))$ by Theorem \ref{thm}.
	
	(3) Since  $G_1(z)$ is harmonic in $\mathbb{C}\backslash R(V)$ according to  Lemma \ref{subharmonic}, by minimum principle of  harmonic function, we have $G_1(z)\neq\log 2$ on $\C\backslash(R(V)\cup {D})$. Therefore $\sigma({H}_{\alpha}(\lambda, \theta))\subset {D}\cup R(V)$ by Theorem \ref{thm}. 
	
	Consequently, we have $\sigma({H}_{\alpha}(\lambda, \theta))= {D}\cup R(V).$

\end{proof}

\begin{proof}[Proof of Example \ref{non-analytic}:]
	We can easily compute $G_1(z)$ using integration by parts,
	
	\begin{align*}
		\int_{0}^{1}\log|z-V(\theta)|\mathrm{d}\theta&=	\int_{0}^{\frac{1}{2}}\log|z-\theta|\mathrm{d}\theta+\int_{\frac{1}{2}}^1\log|z-1+\theta|\mathrm{d}\theta\\
		&=2\int_{0}^{\frac{1}{2}}\log|z-\theta|\mathrm{d}\theta\\
		&=\int_{0}^{\frac{1}{2}}\log((\Re(z)-\theta)^2+\Im^2 (z))\mathrm{d}\theta\\
		&=\int_{\Re(z)}^{\Re(z)-\frac{1}{2}}\log(\theta^2+\Im^2 (z))\mathrm{d}\theta\\	
		&=\theta\log(\theta^2+\Im^2 (z))\bigg|_{\Re(z)}^{\Re(z)-\frac{1}{2}}
		-\int_{\Re(z)}^{\Re(z)-\frac{1}{2}}\frac{2\theta^2}{\theta^2+\Im^2 (z)}\mathrm{d}\theta\\
		&=\theta\log(\theta^2+\Im^2 (z))\bigg|_{\Re(z)}^{\Re(z)-\frac{1}{2}}
		+\int_{\Re(z)}^{\Re(z)+\frac{1}{2}}\frac{2\Im^2 (z)}{\theta^2+\Im^2 (z)}\mathrm{d}\theta-1.
	\end{align*}
	Therefore,

	\begin{equation*}
		\int_{0}^{1}\log|z-V(\theta)|\mathrm{d}\theta=\left\{
		\begin{aligned}
			&2\Re(z)\log(|z|)-(\Re(z)-\frac{1}{2})\log(\Re(z)-\frac{1}{2})^2,\quad \Im(z)=0,\\
			&2\Re(z)\log(|z|)-(\Re(z)-\frac{1}{2})\log((\Re(z)-\frac{1}{2})^2+\Im^2(z))+\\&2\Im(z)(\arctan(\frac{\Im(z)}{2\Im(z)^2+2\Re(z)^2-\Re(z)}))-1, \quad \Im(z)\neq0,\\
		\end{aligned}
		\right.
	\end{equation*}
	where we use the notion $0\cdot\log0=0$.
	Then Example \ref{non-analytic} follows from Theorem \ref{thm}.\end{proof}

	\appendix
\section{Proof of Proposition \ref{res}}
In order to prove Proposition \ref{res}, we only need to show $\widehat{H}^V_{\alpha}(\lambda, \theta)-z I$ is surjective and injective in $z \in \mathbb{C} \backslash S_{\lambda}$. 

\begin{lemma}\label{on}
	Assume $z \in \mathbb{C} \backslash S_{\lambda}$. Then the operator $\widehat{H}^V_{\alpha}(\lambda, \theta)-z I$ is surjective for all  $\theta \in \mathbb{T}^{d}$.
	
\end{lemma}
\begin{proof}
	We distinguish the proof into two cases:\\
	\smallskip
	{\bf Case 1: $G_d(z)>\log\lambda , z \notin  R(V)$.} 
	For any $y=(\dots,y_{-1},y_0,y_1,\dots)\in \ell^{2}(\mathbb{Z})$, we
	take
	$$
	x_{n}=\sum_{m} \alpha_{n, m} y_{m},
	$$
	where
	$$
	\alpha_{n, m}= \begin{cases}-\frac{\lambda^{k}}{\prod_{s=0}^{k}(z-V(\theta+(n-s)) \alpha)}, & \text { if } n-m=k , k \in \mathbb{N}, \\ 0, & \text { otherwise. }\end{cases}
	$$
	
	It is easy to check $\widehat{H}^V_{\alpha}(\lambda, \theta) x-z x=y$. Write $X=\sum_{k \geq 0} X_{k}$, where $X_{k}$ is obtained from $X$ by replacing all entries $\alpha_{n, m}$ for which $n-m\neq k$ by zero. Each $X_{k}$ is bounded since
	\begin{equation*}
		\|X_{k}\|=\sup _{n}|\alpha_{n, n-k}|=\sup _{n} \frac{\lambda^{k}}{\prod_{s=0}^{k}(z- V(\theta+(n-s)\alpha ))}
		\leq \frac{\lambda^k}{\operatorname{dist}(z,R(V))^k} .
	\end{equation*}
	
	By uniquely ergodic theorem, the sequence $\{k^{-1} \log |\alpha_{n, n-k }|\}_{k \geq 1}$ converges uniformly in ${n}$ to $-G_d(z)+\log\lambda<0$ when $k \rightarrow \infty$. This means there exists $k_{0}>0$ such that $\left\|X_{k}\right\|=\sup _{n \in \mathbb{Z}}|\alpha_{n, n-k}| \leq e^{-(G_d(z)-\log\lambda) k / 2}$ for all $k \geq k_{0}$, which shows that $X$ is bounded.\\
	\smallskip		
	{\bf Case 2: $G_d(z)<\log\lambda$.} In this case, take
	$
	x_{n}=\sum_{m} \beta_{n, m} y_{m},
	$
	where
	$$
	\beta_{n, m}= \begin{cases}\lambda^{-1}, & \text { if } m-n=1, \\ \lambda^{-k} \prod_{s=1}^{k-1}(z-V(\theta+(n+s) \alpha), & \text { if } m-n=k , k \geq 2, \\ 0, & \text { otherwise }.\end{cases}
	$$
	Then $\widehat{H}^V_{\alpha}(\lambda, \theta) x-z x=y,$ as well. We write $Y=\sum_{k \geq 1} Y_{k}$, where $Y_{k}$ is obtained from $Y$ by replacing all entries $\beta_{n,m}$ with $m-n \neq k$ by zero. $Y_{k}$ is bounded for each $k$ since
	$$\|Y_{k}\|=\sup _{n}|\beta_{n, n+k }| \leq \frac{(|z|+ \|V\|_{\infty})^{k-1}}{\lambda^{k}}.$$

	By Furman's Theorem \cite{Fur97}, we have
	$$\lim_{k\rightarrow\infty}\sup_{\theta\in\mathbb{T}^d}\frac{1}{k}\log|\beta_{n, n+k}|=G_d(z)-\log\lambda.$$
	
	This shows that there exists $k_{0}>0$ such that for all $k \geq k_{0}$, one has $\left\|Y_{k}\right\| \leq e^{(G_d(z)-\log\lambda) k / 2}$ for $k\geq k_{0}$ and $Y$ is bounded.
\end{proof}

\begin{lemma}\label{in}
	Assume $z \in \mathbb{C} \backslash S_{\lambda}$. Then the following hold.
	\begin{enumerate}
		\item \label{01} If $z \notin R(V)$, the operator $\widehat{H}^V_{\alpha}(\lambda,\theta)- z I$ is injective for all $\theta\in\mathbb{T}^d$.
		\item\label{02} If $z \in R(V)$, there exists $\theta=\theta(z) \in \mathbb{T}^{d}$ such that $\widehat{H}^V_{\alpha}(\lambda, \theta)-$ $z I$ is injective.
	\end{enumerate}
\end{lemma}

\begin{proof}
	For $z \notin R(V)$, by unique ergodic theorem,
	\begin{equation*}\label{erg}
		\lim _{n \rightarrow \infty} \frac{1}{n} \sum_{s=0}^{n} \log |z-V(\theta+s\alpha)| 
		=\lim _{n \rightarrow-\infty} \frac{1}{|n|} \sum_{s=n}^{0} \log |z-V(\theta+s\alpha)|=G_d(z)
	\end{equation*}
	for all $\theta=(\theta_{1}, \dots, \theta_{d}) \in \mathbb{T}^{d}$.
	
	For $z \in R(V)$, we use the Birkhoff ergodic theorem and the fact that a countable intersection of full Lebesgue measure sets has full Lebesgue measure, we have for all $k \in \mathbb{Z}$ and almost all $\theta \in \mathbb{T}^{d}$ :
	\begin{equation}\label{gongshi}
		\begin{aligned}
			&\quad\lim _{n \rightarrow \infty} \frac{1}{n}  \sum_{s=0}^{n-1} \log |z-V(\theta+(k+s)\alpha)| 
			\\&=\lim _{n \rightarrow-\infty} \frac{1}{|n|} \sum_{s=n+1}^{0} \log |z-V(\theta+(k+s)\alpha)|\\&=G_d(z) .
		\end{aligned}
	\end{equation}
	Choose any $\theta$ in case (\ref{01}). In case (\ref{02}), 
	we can take $\theta$  such that \eqref{gongshi} holds for all   $k \in \mathbb{Z}$ .
	
	If $x=(\dots,x_{-1},x_0,x_1,\dots)\in \ell^{2}(\mathbb{Z})\backslash \{0\}$ satisfies $(\widehat{H}^V_{\alpha}(\lambda, \theta)-z I)x=0$, then 
	there exists $k\in\mathbb{Z}$ such that  $x_{k} \neq 0$ and
	$
	x_{k+n}=\beta_{n} x_{k}
	$
	for all $n \in \mathbb{Z}$, where
	$$
	\beta_{n}= \begin{cases}\lambda^{-n} \prod_{s=n-1}^{0}(z-V(\theta+(k+s)\alpha)), & \text { if } n \leq- 1 ,\\ \frac{\lambda^{n}}{\prod_{s=1}^{n}(z-V(\theta+(k+s)\alpha))}, & \text { if }n \geq1.\end{cases}
	$$
	By (\ref{gongshi}) and the choice of $\theta$, one gets
	$$
	\lim _{n \rightarrow \infty} \frac{1}{n} \log |\beta_{n}|=-G_d(z)+\log\lambda=-\lim _{n \rightarrow-\infty} \frac{1}{|n|} \log |\beta_{n}|.
	$$
	
	When $G_d(z)<\log\lambda$, we have $$\lim _{n \rightarrow\infty}|\beta_{n}|=\infty=\lim _{n \rightarrow\infty}|x_{k+n }|= \infty,$$ which contradicts to $x \in \ell^{2}(\mathbb{Z})$. While for $G_d(z)>\log\lambda$, we also have $$\lim _{n \rightarrow -\infty}|\beta_{n}|=\infty=\lim _{n \rightarrow -\infty}|x_{k+n}|,$$ which is again a contradiction.
\end{proof}

\begin{proof}[Proof of Proposition \ref{res}]
	It follows directly from Lemma \ref{on}, Lemma \ref{in} and Lemma  \ref{ind}. 
\end{proof}
\section{Proof of Lemma \ref{app}}
\begin{proof}
	
	We write $T=\sum_{l=-m}^{m}T_l$, where $T_l$ is obtained from $T$ by replacing all entries $t_{i,j}$ for which $i-j\neq l$ by zero. We have $||T_l||\leq M.$ Fix $\delta>0,$ by the  hypothesis, there exists $$0\neq\psi\in \ell^2(\mathbb{Z})$$
	such that $\|\psi\|=1$ and $\|T\psi\|<\delta,$ so that also $\|\sum_{l=1}^{n}T_{\pm l}\psi\|<nM\leq mM$ and $\|(T-\sum_{l=1}^{n}T_{\pm l})\psi\|<\delta+nM\leq\delta+ mM$ for $n\leq m$.
	
	Changing $\psi$ slightly if necessary, we may suppose that $\psi_j\neq0$ for all $j\in \mathbb{Z}$. For each $k$, let
	$$\psi^k=(\ldots, 0,0, \psi_{k+1}, \psi_{k+2}, \ldots, \psi_{k+N}, 0,0, \ldots) \in \ell^2(\mathbb{Z}).$$
	Let us show that, for a suitable choice of $\delta$, for some $k$, $\|T\psi^k\|\leq(\frac{C}{\sqrt{N}}+\varepsilon)\|\psi^k\|.$
	If not, then we have $\|T\psi^k\|>(\frac{C}{\sqrt{N}}+\varepsilon)\|\psi^k\|$ for every $k$,  that is
	$$\begin{aligned}
		\sum_{j=k+1-m}^{k+m}|\sum_{h=k+1-j}^{m}t_{j,j+h}\psi_{j+h}|^2+\sum_{j=k+m+1}^{k+N-m}|\sum_{h=-m}^{m}t_{j,j+h}\psi_{j+h}|^2&\\+\sum_{j=k+N-m+1}^{k+N+m}|\sum_{h=j-(K+N)}^{m}t_{j,j-h}\psi_{j-h}|^2&\\
		>(\frac{C}{\sqrt{N}}+\varepsilon)^2\sum_{h=k+1}^{k+N}|\psi_h|^2.\qquad\qquad\qquad\qquad\qquad\qquad\qquad
	\end{aligned}$$
	
	Summing over $k\in\mathbb{Z}$, we get
	$$2m^3M^2+2m(\delta+m M)^2+(N-2m)\delta^2>(\frac{C}{\sqrt{N}}+\varepsilon)^2N,$$
	thus $$N\delta^2+4\delta m^2M>C^2+\varepsilon^2N+2\sqrt{N}C\varepsilon-4m^3M^2.$$
	Choose $C^2=4m^3M^2,$ then $N\delta^2+4\delta m^2M>\varepsilon^2N+2\sqrt{N}C\varepsilon.$
	Given $\varepsilon,N$, choose $\delta$ sufficient small, which is impossible.
\end{proof}

\section*{Acknowledgements} 
Zhenfu Wang would like to thank Xinrui Tan for useful discussions. This work was partially supported by National Key R\&D Program of China (2020 YFA0713300) and Nankai Zhide Foundation.  J. You was also partially supported by NSFC grant (11871286). Q. Zhou was supported by NSFC grant (12071232), the Science Fund for Distinguished Young Scholars of Tianjin (No. 19JCJQJC61300).

\end{document}